\documentclass[11pt, letterpaper, final]{amsart}
\usepackage{amsfonts,amsmath, amsthm, amssymb, latexsym, epsfig}
\usepackage[english]{babel}
\usepackage[utf8]{inputenc}
\usepackage[all]{xy}
\usepackage{xspace}
\usepackage{comment}
\usepackage{setspace}
\usepackage{enumerate}
\usepackage{stmaryrd}
\usepackage[mathscr]{eucal}
\usepackage[notcite,notref]{showkeys}

	\advance \topmargin by -\headheight
	\advance \topmargin by -\headsep

	\evensidemargin \oddsidemargin
	\newcommand{\Ext}{\ensuremath{\operatorname{Ext}}}
	\newcommand{\multialg}[1]{\mathscr{M}(#1)\xspace}
	\newcommand{\corona}[1]{\mathscr{Q}(#1)\xspace}	
	\newcommand{\Prim}{\ensuremath{\operatorname{Prim}}}
	
	\newcommand{\nuc}{\mathrm{nuc}}

	\theoremstyle{plain}
	\newtheorem{thm}{Theorem}[section]
	\newtheorem{lemma}[thm]{Lemma}
	\newtheorem{theorem}[thm]{Theorem}
	\newtheorem{proposition}[thm]{Proposition}
	\newtheorem{corollary}[thm]{Corollary}

	\theoremstyle{definition}
	\newtheorem{definition}[thm]{Definition}
	
	\newtheorem{remark}[thm]{Remark}
	\newtheorem{notation}[thm]{Notation}
	\newtheorem{example}[thm]{Example}

	\newtheorem{observation}[thm]{Observation}
	
	\numberwithin{equation}{section}
	\numberwithin{figure}{section}

\begin{document}
	\title{Lifting theorems for completely positive maps}
	\author{James Gabe}
	     \address{Current address: Department of Mathematics and Computer Science, University of Southern Denmark, Odense, Denmark.}
        \address{Mathematical Sciences \\
        University of Southampton\\
        University Road\\
        Southampton SO17 1BJ, United Kingdom}
        \email{gabe@imada.sdu.dk}
        \thanks{This work was supported by the Danish National Research Foundation through the Centre for Symmetry and Deformation (DNRF92).}

\begin{abstract}
We prove lifting theorems for completely positive maps going out of exact $C^\ast$-algebras, where we remain in control of which ideals are mapped into which.
A consequence is, that if $\mathsf X$ is a second countable topological space, $\mathfrak A$ and $\mathfrak B$ are separable, nuclear $C^\ast$-algebras over $\mathsf X$, 
and the action of $\mathsf X$ on $\mathfrak A$ is continuous, 
then $E(\mathsf X; \mathfrak A, \mathfrak B) \cong KK(\mathsf X; \mathfrak A, \mathfrak B)$ naturally.

As an application, we show that a separable, nuclear, strongly purely infinite $C^\ast$-algebra $\mathfrak A$ absorbs a strongly self-absorbing $C^\ast$-algebra $\mathscr D$ if and only if $\mathfrak I$ and $\mathfrak I\otimes \mathscr D$ are $KK$-equivalent for every two-sided, closed ideal $\mathfrak I$ in $\mathfrak A$. In particular, if $\mathfrak A$ is separable, nuclear, and strongly purely infinite, then $\mathfrak A \otimes \mathcal O_2 \cong \mathfrak A$ if and only if every two-sided, closed ideal in $\mathfrak A$ is $KK$-equivalent to zero.
\end{abstract}

\maketitle 

\section{Introduction}

Arveson was perhaps the first to recognise the importance of lifting theorems for completely positive maps. 
In \cite{Arveson-essnormal}, he uses a lifting theorem to give a simple and operator theoretic proof of the fact that the Brown--Douglas--Fillmore semigroup $\Ext(\mathsf X)$ is actually a group.
This was already proved by Brown, Douglas, and Fillmore in \cite{BrownDouglasFillmore-unitaryeq}, but the proof was somewhat complicated and very topological in nature.
All the known lifting theorems at that time were generalised by Choi and Effros \cite{ChoiEffros-lifting}, when they proved that any nuclear map going out of a separable $C^\ast$-algebra is liftable.
This result, together with the dilation theorem of Stinespring \cite{Stinespring-positive} and the Weyl--von Neumann type theorem of Voiculescu \cite{Voiculescu-WvN}, was used by Arveson \cite{Arveson-extensions}
to prove that the (generalised) Brown--Douglas--Fillmore semigroup $\Ext(\mathfrak A)$ defined in \cite{BrownDouglasFillmore-Ext} is a group for any unital, separable, nuclear $C^\ast$-algebra $\mathfrak A$.
When doing this, Arveson included a simplified proof of the lifting theorem of Choi and Effros, a proof which in many ways illustrates, that the Choi--Effros lifting theorem is a non-commutative analogue of the selection theorems of Michael \cite{Michael-selection}.

Kasparov \cite{Kasparov-KKExt} used the same idea as Arveson to prove that for any separable, nuclear $C^\ast$-algebra $\mathfrak A$ and any $\sigma$-unital $C^\ast$-algebra $\mathfrak B$, the semigroup $\Ext(\mathfrak A, \mathfrak B)$ is in fact a group.
It was also an application of the Choi--Effros lifting theorem, which allowed Kasparov to prove that the functor $KK(\mathfrak A, -)$ is half-exact for any separable, nuclear $C^\ast$-algebra $\mathfrak A$, and thus induces a six-term exact sequence for any 
short exact sequence of $\sigma$-unital $C^\ast$-algebras. 
This fails if one does not assume $\mathfrak A$ to be nuclear, which is basically due to the fact, that we can not lift completely positive maps in general.
So $KK$-theory lacks certain desirable properties such as excision, i.e.~that short exact sequences of $C^\ast$-algebras induce six-term exact sequences of $KK$-groups.
In an attempt to fix this ``defect'' of $KK$-theory, Higson \cite{Higson-catfrac} constructed $E$-theory, which resembles $KK$-theory quite a bit, but which is always half-exact.
As a consequence of the half-exactness of $KK(\mathfrak A, -)$ for separable, nuclear $C^\ast$-algebras $\mathfrak A$, it follows that $E(\mathfrak A,\mathfrak B) \cong KK(\mathfrak A, \mathfrak B)$ naturally, for such $\mathfrak A$.

We say that a topological space $\mathsf X$ acts on a $C^\ast$-algebra $\mathfrak A$, if there is an order preserving map from the lattice $\mathbb O(\mathsf X)$ of open subsets of $\mathsf X$,
to the lattice $\mathbb I(\mathfrak A)$ of two-sided, closed ideals in $\mathfrak A$. A map between such $C^\ast$-algebras is $\mathsf X$-equivariant, if it respects the action.
Kirchberg introduced a modified version of $KK$-theory for $C^\ast$-algebras with an action of $\mathsf X$, and proved the very deep result \cite{Kirchberg-non-simple} (see also \cite{Gabe-Oinfty}), that all separable, nuclear, strongly purely infinite 
$C^\ast$-algebras with a tight action of $\mathsf X$, are classified by $KK(\mathsf X)$-theory. Here tight refers to the action $\mathbb O(\mathsf X) \to \mathbb I(\mathfrak A)$ being a lattice isomorphism.
As it turns out, the functor $KK(\mathsf X; \mathfrak A, -)$ is not half-exact in general, not even when $\mathfrak A$ is nuclear.
This is mainly due to the lack of lifting theorems for completely positive maps, for which we preserve the action of $\mathsf X$.

In \cite{DadarlatMeyer-E-theory}, Dadarlat and Meyer construct a version of $E$-theory for $C^\ast$-algebras with an action of $\mathsf X$, which is half-exact, and which also possesses other nice properties which $KK(\mathsf X)$-theory does not enjoy.
Thus it would be desirable to find sufficient criteria for when $E(\mathsf X; \mathfrak A, \mathfrak B) \cong KK(\mathsf X; \mathfrak A, \mathfrak B)$ naturally, 
as it is known that nuclearity of $\mathfrak A$ does not suffice. 
This is the main motivation of this paper. We show that if $\mathfrak A$ and $\mathfrak B$ are nuclear, and if the actions of $\mathsf X$ on $\mathfrak A$ and $\mathfrak B$ satisfy certain continuity properties, then
$E(\mathsf X; \mathfrak A, \mathfrak B)\cong KK(\mathsf X; \mathfrak A, \mathfrak B)$ naturally.
This is done by proving that we may lift $\mathsf X$-equivariant completely positive maps, while preserving the $\mathsf X$-equivariant structure.

Combining this result with the deep classification result of Kirchberg \cite{Kirchberg-non-simple}, it follows that all separable, nuclear, strongly purely infinite $C^\ast$-algebras with a tight action of $\mathsf X$, are classified by $E(\mathsf X)$-theory. 
We apply this to show that if $\mathfrak A$ is a separable, nuclear, strongly purely infinite $C^\ast$-algebra, and $\mathscr D$ is a strongly self-absorbing $C^\ast$-algebra, then $\mathfrak A \cong \mathfrak A \otimes \mathscr D$ if and only if $\mathfrak I$ and $\mathfrak I\otimes \mathscr D$ are $KK$-equivalent for every two-sided, closed ideal $\mathfrak I$ in $\mathfrak A$.

In particular, let $\mathfrak A$ be a separable, nuclear, strongly purely infinite $C^\ast$-algebra, let $M_{n^\infty}$ denote the UHF algebra of type $n^{\infty}$, $\mathcal Q$ denote the universal UHF algebra, and $\mathcal O_2$ denote the Cuntz algebra. We show that:
\begin{itemize}
\item If all two-sided, closed ideals in $\mathfrak A$ satisfy the UCT, then $\mathfrak A \cong \mathfrak A \otimes M_{n^\infty}$ if and only if $K_\ast(\mathfrak I)$ is uniquely $n$-divisible for every two-sided, closed ideal $\mathfrak I$ in $\mathfrak A$.
\item If all two-sided, closed ideals in $\mathfrak A$ satisfy the UCT, then $\mathfrak A \cong \mathfrak A \otimes \mathcal Q$ if and only if $K_\ast(\mathfrak I)$ is uniquely divisible for every two-sided, closed ideal $\mathfrak I$ in $\mathfrak A$.
\item $\mathfrak A \cong \mathfrak A \otimes \mathcal O_2$ if and only if every two-sided, closed ideal in $\mathfrak A$ is $KK$-equivalent to zero.
\end{itemize}

The author has been made aware that Eberhard Kirchberg announced results partially overlapping with results presented here, at the 2009 Oberwolfach meeting ``C*-Algebren'', cf.~\cite{Kirchberg-OWR}, and thanks Ralf Meyer for pointing this out.

\section{A Hahn--Banach separation theorem for closed operator convex cones}

\begin{definition}\label{d:opconcon}
Let $\mathfrak A$ and $\mathfrak B$ be $C^\ast$-algebras and let $CP(\mathfrak A,\mathfrak B)$ denote the convex cone of all completely positive (c.p.) maps from $\mathfrak A$ to $\mathfrak B$. 
A subset $\mathscr C$ of $CP(\mathfrak A,\mathfrak B)$ is called an \emph{operator convex cone} if it satisfies the following:
\begin{itemize}
\item[$(1)$] $\mathscr C$ is a convex cone,
\item[$(2)$] If $\phi \in \mathscr C$ and $b$ in $\mathfrak B$ then $b^\ast \phi(-) b \in \mathscr C$,
\item[$(3)$] If $\phi \in \mathscr C$, $a_1,\dots,a_n \in \mathfrak A$, and $b_1,\dots,b_n \in \mathfrak B$ then the map
\begin{equation}\label{eq:basicmap}
\sum_{i,j=1}^n b_i^\ast \phi(a_i^\ast (-) a_j)b_j
\end{equation}
is in $\mathscr C$.
\end{itemize}
We equip $\mathscr C$ with the point-norm topology, and say that it is a \emph{closed} operator convex cone, if it is closed as a subspace of $CP(\mathfrak A,\mathfrak B)$.
\end{definition}

We will almost only be considering operator convex cones which are closed.

\begin{example}\label{e:nuclearmaps}
 A c.p.~map is called \emph{factorable} if it factors through a matrix algebra by c.p.~maps. The set $CP_\mathrm{fact}(\mathfrak A, \mathfrak B)\subseteq CP(\mathfrak A, \mathfrak B)$ of all factorable maps is an operator convex cone.

Checking $(1)$ in the definition amounts to the observation, that there exists a conditional expectation $M_{k+l} \to M_k \oplus M_l$, so if two c.p.~maps factor through $M_k$ and $M_l$ respectively, then their sum factors through $M_{k+l}$. Condition $(2)$ is obvious, so only $(3)$ remains to be checked. Let $a_1,\dots, a_n\in \mathfrak A$ and $b_1,\dots,b_n\in \mathfrak B$ be given, and $\phi^{(n)} \colon M_n(\mathfrak A) \to M_n(\mathfrak B)$ be the amplification of $\phi$. Let  $r \in M_{1,n}(\mathfrak A)$ be the row vector $r= (a_1 \; \dots \; a_n)$, and $c\in M_{n,1}(\mathfrak B)$ be the column vector $(b_1 \; \dots \; b_n)^t$. The map in equation \eqref{eq:basicmap} is exactly $c^\ast \phi^{(n)} (r^\ast (-) r) c$, which is factorable since $\phi^{(n)}$ is (clearly) factorable. Hence  $CP_\mathrm{fact}(\mathfrak A, \mathfrak B)$ is an operator convex cone.
 
 A c.p.~map is called \emph{nuclear} if it can be approximated point-norm by factorable maps, i.e.~if it is in the point-norm closure of $CP_\mathrm{fact}(\mathfrak A, \mathfrak B)$. 
 Thus the set $CP_\nuc(\mathfrak A,\mathfrak B)$ of nuclear c.p.~maps is a closed operator convex cone.
\end{example}

The above definition of nuclearity agrees with the one often used in the literature (for contractive maps), e.g.~the definition used in the book by Brown and Ozawa \cite[Definition 2.1.1]{BrownOzawa-book-approx}, in which the maps going in and out of the matrix algebras are assumed to be contractive.
This has been well-known for a long time, and a proof of this is presented in \cite[Lemma 2.3]{GabeRuiz-absrep} (alternatively, see \cite[Lemma 3.7]{Gabe-Oinfty}).

\begin{observation}
 Note that by our definition of a nuclear map, it follows immediately that the composition of any c.p.~map with a nuclear map, is again nuclear. We will use this fact several times without mentioning.
\end{observation}

The following is a well-known, very basic result on c.p.~maps using the Hahn--Banach separation theorem. 

\begin{lemma}\label{l:cpHahnBanach}
 Let $\mathscr C \subseteq CP(\mathfrak A, \mathfrak B)$ be a point-norm closed convex subset. If $\phi \in CP(\mathfrak A, \mathfrak B)$ is in the point-weak closure $\overline{\mathscr C}^{pt\mathrm{-}weak} \subseteq CP(\mathfrak A, \mathfrak B^{\ast \ast})$, then
 $\phi \in \mathscr C$, i.e.~if for every $a_1, \dots , a_n\in \mathfrak A$, every $f_1, \dots, f_n \in \mathfrak B^\ast$ (or in the state space $S(\mathfrak B)$) and every $\epsilon >0$ there is a $\psi \in \mathscr C$, such that
 \[
  | f_i(\phi(a_i)) - f_i(\psi(a_i))| < \epsilon, \text{ for }i = 1,\dots ,n,
 \]
 then $\phi \in \mathscr C$.
\end{lemma}
\begin{proof}
 This is an easy Hahn--Banach separation argument. In fact, let $a_1, \dots, a_n \in \mathfrak A$. The set
 \[
  \{ (\psi(a_1), \dots, \psi(a_n)) : \psi \in \mathscr C\}
 \]
 is a norm-closed convex subset of $\mathfrak B^n$. Hence by the Hahn--Banach separation theorem (since we can not separate $(\phi(a_1), \dots, \phi(a_n))$ from the above set by linear functionals) we must have $(\phi(a_1), \dots , \phi(a_n))$ is in the above set.
 Now the result follows trivially since $\mathscr C$ is point-norm closed.
\end{proof}

Kirchberg and Rørdam show in \cite[Proposition 4.2]{KirchbergRordam-zero}, that if $\mathscr C \subseteq CP(\mathfrak A,\mathfrak B)$ is a closed operator convex cone, where $\mathfrak A$ is separable and nuclear, and $\phi \colon \mathfrak A\to \mathfrak B$ is any c.p.~map, then $\phi \in \mathscr C$ if and only if $\phi(a) \in \overline{\mathfrak B \{\psi(a):\psi \in \mathscr C\} \mathfrak B}$ for every $a\in \mathfrak A$. We refer to this as a Hahn--Banach separation theorem for closed operator convex cones, as one obtains a separation of $\phi$ from a closed operator convex cone, and since the result relies heavily on the Hahn--Banach separation theorem.

We generalise the result of Kirchberg and Rørdam to exact $C^\ast$-algebras and nuclear c.p.~maps, and where we only take positive elements in $\mathfrak A$. The proof is virtually identical to the proof of \cite[Proposition 4.2]{KirchbergRordam-zero}, but we fill in the proof for completion.

\begin{theorem}[Cf.~Kirchberg--Rørdam]\label{t:uniquecone}
 Let $\mathfrak A$ and $\mathfrak B$ be $C^\ast$-algebras with $\mathfrak A$ exact, and let $\mathscr C \subseteq CP(\mathfrak A, \mathfrak B)$ be a closed operator convex cone. 
 Suppose that $\mathscr C \subseteq CP_\nuc(\mathfrak A, \mathfrak B)$ and let $\phi \in CP_\nuc(\mathfrak A, \mathfrak B)$. 
 Then $\phi \in \mathscr C$ if and only if $\phi(a) \in \overline{\mathfrak B \{ \psi (a) : \psi \in \mathscr C\} \mathfrak B}$ for every positive $a\in \mathfrak A$.
\end{theorem}
\begin{proof}
``Only if'' is obvious. For ``if'', suppose $\phi(a) \in \overline{\mathfrak B \{ \psi (a) : \psi \in \mathscr C\} \mathfrak B}$ for every positive $a\in \mathfrak A$.
 By Lemma \ref{l:cpHahnBanach} it suffices to show, that given $a_1,\dots,a_n\in \mathfrak A$, $\epsilon >0$ and $f_1,\dots,f_n\in \mathfrak B^\ast$, there is a $\psi \in \mathscr C$ such that
 \[
  | f_i(\phi(a_i)) - f_i(\psi(a_i))| < \epsilon, \quad \text{for } i=1,\dots, n.
 \]
 By \cite[Lemma 7.17 (i)]{KirchbergRordam-absorbingOinfty} we may find a cyclic representation $\pi \colon \mathfrak B \to \mathbb B(\mathcal H)$ with cyclic vector $\xi \in \mathcal H$, and elements $c_1, \dots, c_n \in \pi(\mathfrak B)' \cap \mathbb B(\mathcal H)$,
 such that $f_i(b) = \langle \pi(b) c_i \xi, \xi\rangle$ for $i=1,\dots, n$. Let $\mathfrak C = C^\ast(c_1,\dots,c_n)$ and $\iota \colon \mathfrak C \hookrightarrow \mathbb B(\mathcal H)$ be the inclusion.
 For any c.p.~map $\rho \colon \mathfrak A \to \mathfrak B$ there is an induced positive linear functional on $\mathfrak A \otimes_{\max{}} \mathfrak C$ given by the composition
 \[
  \mathfrak A \otimes_{\max{}} \mathfrak C \xrightarrow{\rho \otimes id_\mathfrak{C}} \mathfrak B \otimes_{\max{}} \mathfrak C \xrightarrow{\pi \times \iota} \mathbb B(\mathcal H) \xrightarrow{ \omega_\xi} \mathbb C,
 \]
 where $\omega_\xi$ is the vector functional induced by $\xi$, i.e.~$\omega_\xi(T) = \langle T\xi , \xi\rangle$. 
 If $\rho$ is nuclear, then $\rho \otimes id_\mathfrak{C}$ above factors through the spatial tensor product $\mathfrak A \otimes \mathfrak C$ (see e.g.~\cite[Lemma 3.6.10]{BrownOzawa-book-approx}),
 so if $\rho$ is nuclear it induces a positive linear functional $\eta_\rho$ on $\mathfrak A \otimes \mathfrak C$. 
 
 Let $\mathscr K$ be the weak-$\ast$ closure of $\{ \eta_{\psi} : \psi \in \mathscr C\} \subseteq (\mathfrak A \otimes \mathfrak C)^\ast$. 
 It suffices to show that $\eta_\phi \in \mathscr K$ since, if $|\eta_\phi(a_i \otimes c_i) -  \eta_\psi(a_i \otimes c_i)| < \epsilon$ for some $\psi \in \mathscr C$, then
 \[
  f_i(\phi(a_i)) = \langle \pi(\phi(a_i)) c_i \xi , \xi \rangle = \eta_\phi(a_i \otimes c_i) \approx_\epsilon \eta_\psi(a_i \otimes c_i) = \langle \pi(\psi(a_i)) c_i \xi , \xi \rangle = f_i(\psi(a_i)),
 \]
 for $i=1,\dots, n$, which is what we want to prove. It is easily verified (e.g.~by checking on elementary tensors $a\otimes c$) that $\eta_{\psi_1} + \eta_{\psi_2} = \eta_{\psi_1 + \psi_2}$, 
 and that $t \eta_\psi = \eta_{t\psi}$ for $t \in \mathbb R_+$. Hence $\mathscr K$ is a weak-$\ast$ closed convex cone of positive linear functionals.
 
 We want to show that if $\eta \in \mathscr K$ and $d\in \mathfrak A \otimes \mathfrak C$, then $d^\ast \eta d := \eta(d^\ast(-)d) \in \mathscr K$.
 Since $\mathscr K$ is weak-$\ast$ closed, it suffices to show this for $\eta = \eta_\psi$ where $\psi \in \mathscr C$, and $d= \sum_{j=1}^k x_j \otimes y_j$ 
 where $x_1,\dots,x_k\in \mathfrak A$ and $y_1 , \dots,y_k\in \mathfrak C$. For $a\in \mathfrak A$ and $c \in \mathfrak C$ we have
 \[
  \eta_{\psi}(d^\ast(a\otimes c)d) = \sum_{j,l=1}^k \eta_{\psi}((x_j^\ast a x_l) \otimes (y_j^\ast c y_l)) = \sum_{j,l=1}^k \langle \pi(\psi(x_j^\ast a x_l)) c y_l \xi, y_j \xi\rangle.
 \]
 Since $\xi$ is cyclic for $\pi$ we may, for any $\delta >0$, find $b_1,\dots,b_k\in \mathfrak B$ such that $\| \pi(b_j)\xi - y_j \xi \| < \delta$. Thus by choosing $\delta$ sufficiently small we may approximate $d^\ast \eta_\psi d$ in the weak-$\ast$ topology by
 \[
  \sum_{j,l=1}^k \langle \pi(\psi(x_j^\ast a x_l)) c \pi(b_l) \xi, \pi(b_j) \xi\rangle = \langle \pi(\sum_{j,l=1}^k b_j^\ast\psi(x_j^\ast a x_l) b_l) c \xi , \xi\rangle = \eta_{\psi_0}(a\otimes c)
 \]
 where $\psi_0 = \sum_{j,l=1}^k b_j^\ast \psi(x_j^\ast (-)x_l)b_l \in \mathscr C$. Thus $d^\ast \eta d \in \mathscr K$ for any $\eta \in \mathscr K$ and $d\in \mathfrak A \otimes \mathfrak C$. 
 Let $\mathfrak J$ be the subset of $\mathfrak A \otimes \mathfrak C$ consisting of elements $d$ such that $\eta(d^\ast d) = 0$ for all $\eta \in \mathscr K$. 
 By \cite[Lemma 7.17 (ii)]{KirchbergRordam-absorbingOinfty} it follows that $\mathfrak J$ is a closed two-sided ideal in $\mathfrak A \otimes \mathfrak C$, and that $\eta_\phi \in \mathscr K$ if $\eta_\phi(d^\ast d) = 0$ for all $d\in \mathfrak J$.
 
 Since $\mathfrak A$ is exact, $\mathfrak J$ is the closed linear span of all elementary tensors $x\otimes y$ for which $x\otimes y \in \mathfrak J$ (see e.g.~\cite[Corollary 9.4.6]{BrownOzawa-book-approx}).
 Recall that the left kernel of $\eta_\phi$, i.e.~the set of all $d$ such that $\eta_\phi(d^\ast d) = 0$, is a closed linear subspace of $\mathfrak A \otimes \mathfrak C$. Hence it suffices to show, that when $x\in \mathfrak A$ and $y \in \mathfrak C$ are
 such that $x\otimes y\in \mathfrak J$, then $\eta_\phi(x^\ast x \otimes y^\ast y ) = 0$. Fix such $x$ and $y$.
 
 By assumption $\phi(x^\ast x) \in \overline{\mathfrak B\{ \psi (x^\ast x) : \psi \in \mathscr C\}\mathfrak B}$. Thus for any $\delta>0$ we may choose $\psi_1,\dots , \psi_m \in \mathscr C$ and $b_1,\dots,b_m \in \mathfrak B$ such that
 \[
  \| \phi(x^\ast x) - (\sum_{j=1}^m b_j^\ast \psi_j(x^\ast x) b_j)\| < \delta.
 \]
 Let $\psi = \sum_{j=1}^m b_j^\ast \psi_j(-) b_j$ which is in $\mathscr C$, so that $\| \phi(x^\ast x) - \psi(x^\ast x)\| < \delta$. Since $\eta_\psi(x^\ast x \otimes y^\ast y) = 0$ we get that
 \begin{eqnarray*}
  |\eta_\phi(x^\ast x \otimes y^\ast y) | &=& | \eta_\phi(x^\ast x \otimes y^\ast y) - \eta_\psi(x^\ast x \otimes y^\ast y)| \\
  &=& | \langle \pi(\phi(x^\ast x)-\psi(x^\ast x)) y\xi, y \xi\rangle | \\
  &<& \delta \| y\xi\|^2.
 \end{eqnarray*}
 Since $\delta$ was arbitrary we get that $\eta_\phi(x^\ast x \otimes y^\ast y) = 0$ which finishes the proof.
\end{proof}


\subsection{An abstract lifting result}

The main goal of this paper, is to obtain lifting results for c.p.~maps, where we remain in control of the lift, in the sense that we may choose a lift in a given closed operator convex cone. This can be obtained as an application of the Hahn--Banach type theorem. First we need a lemma, which is essentially due to Arveson.

\begin{lemma}\label{l:arveson}
Let $\mathfrak A, \mathfrak B$ and $\mathfrak C$ be $C^\ast$-algebras with $\mathfrak A$ separable, and let $\pi \colon \mathfrak B \to \mathfrak C$ be a surjective $\ast$-homomorphism.
Let $\mathscr C \subseteq CP(\mathfrak A,\mathfrak B)$ be a closed operator convex cone. Then
\[
 \pi (\mathscr C) := \{ \pi \circ \psi \mid \psi \in \mathscr C\}
\]
is a closed operator convex cone.
\end{lemma}
\begin{proof}
 Clearly $\pi(\mathscr C)$ is an operator convex cone. That $\pi(\mathscr C)$ is point-norm closed is essentially the same proof as \cite[Theorem 6]{Arveson-extensions} (that the set of c.p.~maps with contractive c.p.~lifts is point-norm closed). 
However, to run Arveson's argument we must show that if $\phi\colon \mathfrak A \to \mathfrak C$ is a contractive c.p.~map which is a point-norm limit of a net of (not necessarily contractive) maps $\pi \circ \psi_\lambda$ with $\psi_\lambda\in \mathscr C$, then there is a sequence of contractive maps $\tilde \psi_n \in \mathscr C$ such that $\pi \circ \tilde \psi_n \to \phi$ point-norm.
Let $(a_n)_{n\in \mathbb N}$ and $(e_n)_{n\in \mathbb N}$ be a dense sequence and an approximate identity respectively in $\mathfrak A$. For each $n\in \mathbb N$ we fix $\lambda_n$ such that $\| \pi(\psi_{\lambda_n}(e_n x e_n)) - \phi(e_n x e_n)\| < 1/n$ for $x\in \{ 1, a_1,\dots, a_n\}$. We may pick a positive contraction $f_n \in \ker \pi$ such that $\| (1-f_n) \psi_{\lambda_n}(e_n^2) (1-f_n)\| < \|\phi(e_n^2)\| + 1/n \leq \tfrac{n+1}{n}$. Let $\tilde\psi_n := \tfrac{n}{n+1} (1-f_n)\psi_{\lambda_n}(e_n(-) e_n) (1-f_n) \in \mathscr C$ which is contractive. It is easy to check that $\pi \circ \tilde \psi_n \to \phi$ point-norm. Now the exact same proof as \cite[Theorem 6]{Arveson-extensions} (alternatively, see \cite[Lemma C.2]{BrownOzawa-book-approx}) provides $\psi \in \mathscr C$ such that $\pi \circ \psi = \phi$.
\end{proof}

\begin{proposition}\label{p:liftingconeversion}
Let $\mathfrak A$ be a separable, exact $C^\ast$-algebra, let $\mathfrak B$ be a $C^\ast$-algebra with a two-sided, closed ideal $\mathfrak J$, and let $\pi \colon \mathfrak B \to \mathfrak B/\mathfrak J$ be the quotient map. Let $\mathscr C \subseteq CP_\nuc(\mathfrak A, \mathfrak B)$ be a closed operator convex cone. A c.p.~map $\phi \colon \mathfrak A \to \mathfrak B/\mathfrak J$ lifts to a c.p.~map in $\mathscr C$ if and only if $\phi$ is nuclear and $\phi(a) \in \pi(\overline{\mathfrak B \{ \psi(a): \psi \in \mathscr C\}\mathfrak B})$, for every positive $a\in \mathfrak A$.
\end{proposition}
\begin{proof}
If $\phi$ lifts to a map $\psi\in \mathscr C$, then $\phi = \pi \circ \psi$ is nuclear (as $\psi$ is nuclear), and $\phi(a) = \pi(\psi(a)) \in \pi(\overline{\mathfrak B \{ \psi'(a) : \psi'\in \mathscr C\} \mathfrak B})$ for every positive $a\in \mathfrak A$.

Suppose that $\phi$ is nuclear, and that $\phi(a) \in \pi(\overline{\mathfrak B \{ \psi(a) : \psi\in \mathscr C\} \mathfrak B})$ for every positive $a\in \mathfrak A$. By Lemma \ref{l:arveson}, the set $\pi (\mathscr C)$ of c.p.~maps that lift to $\mathscr C$, is a closed operator convex cone consisting only of nuclear maps. Thus by our Hahn--Banach type separation theorem (Theorem \ref{t:uniquecone}), $\phi \in \pi(\mathscr C)$ if and only if 
\begin{eqnarray*}
\phi(a) &\in& \overline{(\mathfrak B/\mathfrak J) \{ \eta (a) : \eta \in  \pi(\mathscr C)\} (\mathfrak B/\mathfrak J)} \\
&=& \overline{\pi(\mathfrak B) \{ \pi \circ \psi(a) :  \psi \in \mathscr C\} \pi(\mathfrak B)} \\
&=& \pi \left( \overline{ \mathfrak B \{ \psi(a) : \psi \in \mathscr C\} \mathfrak B} \right).
\end{eqnarray*}
\end{proof}


\section{Exact $C^\ast$-algebras and nuclear maps}

In this short section we prove a few well-known results about exact $C^\ast$-algebras. For a $C^\ast$-algebra $\mathfrak B$, we let $\multialg{\mathfrak B}$ denote its multiplier algebra, and $\corona{\mathfrak B} := \multialg{\mathfrak B}/\mathfrak B$ its corona algebra.

\begin{definition}
Let $\mathfrak A$ and $\mathfrak B$ be $C^\ast$-algebras, and $\phi \colon \mathfrak A \to \multialg{\mathfrak B}$ be a c.p.~map. We say that $\phi$ is \emph{weakly nuclear} if the c.p.~maps $b^\ast \phi(-) b \colon \mathfrak A \to \mathfrak B$ are nuclear for all $b\in \mathfrak B$.
\end{definition}

Recall, that a $C^\ast$-algebra is exact if and only if the $C^\ast$-algebra has a faithful representation on a Hilbert space which is nuclear.
By Arveson's extension theorem, this is equivalent to \emph{any} represenation on a Hilbert space being nuclear.
We need the following other characterisation of exactness.

\begin{proposition}\label{p:exactvsweaklynuc}
 A $C^\ast$-algebra $\mathfrak A$ is exact if and only if it holds that for any $\sigma$-unital $C^\ast$-algebra $\mathfrak B$ and any weakly nuclear map $\phi \colon \mathfrak A \to \multialg{\mathfrak B}$, $\phi$ is nuclear.
\end{proposition}
\begin{proof}
 Suppose that any weakly nuclear map from $\mathfrak A$ into a multiplier algebra of a $\sigma$-unital $C^\ast$-algebra is nuclear. To show that $\mathfrak A$ is exact, it suffices to show that every separable $C^\ast$-subalgebra is exact.
 Let $\mathfrak A_0\subseteq \mathfrak A$ be a separable $C^\ast$-subalgebra, and let $\pi \colon \mathfrak A_0 \to \multialg{\mathbb K}$ be a faithful representation. 
 By Arveson's extension theorem, we may extend this map to a c.p.~map $\tilde \pi \colon \mathfrak A \to \multialg{\mathbb K}$, which is nuclear by assumption. Thus $\pi$ is nuclear and hence $\mathfrak A_0$ is exact. It follows that $\mathfrak A$ is exact.
 
 Now suppose that $\mathfrak A$ is exact, that $\mathfrak B$ is any $\sigma$-unital $C^\ast$-algebra and $\phi \colon \mathfrak A \to \multialg{\mathfrak B}$ is weakly nuclear.
 By standard arguments we may assume that $\mathfrak A$ and $\phi$ are unital. It suffices to show that for any unital, separable $C^\ast$-subalgebra $\mathfrak A_0$ the restriction $\phi|_{\mathfrak{A}_0}$ is nuclear.
 Let $\iota \colon \mathfrak A_0 \hookrightarrow \multialg{\mathbb K}$ be a unital inclusion. Since $\mathfrak A_0$ is a $C^\ast$-subalgebra of an exact $C^\ast$-algebra, it is itself exact, and thus $\iota$ is nuclear.
 Let $\Phi$ be the composition
 \[
  \mathfrak A_0 \xrightarrow{\iota} \multialg{\mathbb K} \xrightarrow{1 \otimes id} \multialg{\mathfrak B} \otimes \multialg{\mathbb K} \hookrightarrow \multialg{\mathfrak B \otimes \mathbb K},
 \]
 which is nuclear. It basically follows from a result of Kasparov in \cite{Kasparov-Stinespring} (see \cite{DadarlatEilers-classification} for details on generalising Kasparov's result to the case which we are considering) 
 that $\Phi$ absorbs any unital weakly nuclear c.p.~map.
 In particular, it absorbs the map $\phi_0$, defined as the composition
 \[
  \mathfrak A_0 \xrightarrow{\phi|_{\mathfrak A_0}} \multialg{\mathfrak B} \xrightarrow{id \otimes e_{11}} \multialg{\mathfrak B} \otimes \multialg{\mathbb K} \hookrightarrow \multialg{\mathfrak B \otimes \mathbb K}.
 \]
 Thus there is a sequence of isometries $(v_n)$ in $\multialg{\mathfrak B \otimes \mathbb K}$ such that $v_n^\ast \Phi(-) v_n$ converges point-norm to $\phi_0$. Since $\Phi$ is nuclear it follows that $\phi_0$ is nuclear.
 There is a conditional expectation $\Psi$ given by the composition
 \[
  \multialg{\mathfrak B \otimes \mathbb K} \xrightarrow{(1 \otimes e_{11})(-)(1 \otimes e_{11})} \multialg{\mathfrak B} \otimes e_{11} \cong \multialg{\mathfrak B}
 \]
 such that $\Psi \circ \phi_0 = \phi|_{\mathfrak A_0}$ and thus $\phi|_{\mathfrak A_0}$ is nuclear.
\end{proof}

Recall, that when $0 \to \mathfrak B \to \mathfrak E \to \mathfrak A \to 0$ is an extension of $C^\ast$-algebras, 
there is an induced $\ast$-homomorphism $\tau \colon \mathfrak A \to \corona{\mathfrak B} := \multialg{\mathfrak B}/\mathfrak B$ called the \emph{Busby map}.
Also, there is a canonical isomorphism from $\mathfrak E$ onto the pull-back 
\[
 \mathfrak A \oplus_{\corona{\mathfrak B}} \multialg{\mathfrak B} := \{ (a,m) \in \mathfrak A \oplus \multialg{\mathfrak B} : \tau(a) = m + \mathfrak B\}.
\]

An interesting observation can be made on extensions of exact $C^\ast$-algebras by nuclear $C^\ast$-algebras. This will be used in Theorem \ref{t:Xlifting} to prove an Effros--Haagerup type lifting result, cf.~\cite{EffrosHaagerup-lifting}.

\begin{corollary}\label{c:exactext}
 Let $0 \to \mathfrak B \to \mathfrak E \to \mathfrak A \to 0$ be an extension of $C^\ast$-algebras with Busby map $\tau$. Suppose that $\mathfrak A$ is exact and $\mathfrak B$ is $\sigma$-unital and nuclear. 
 Then $\mathfrak E$ is exact if and only if $\tau$ is nuclear.
\end{corollary}
\begin{proof}
 If $\mathfrak E$ is non-unital we may consider the unitised extension $0 \to \mathfrak B \to \mathfrak E^\dagger \to \mathfrak A^\dagger \to 0$. Since $\tau$ is nuclear if and only if the unitisation $\tau^\dagger$ is nuclear, and $\mathfrak E$ is exact if
 and only if $\mathfrak E^\dagger$ is exact, we may assume that $\mathfrak E$ is unital.
 It is well-known (see e.g.~\cite[Exercise 3.9.8]{BrownOzawa-book-approx}) that the extension algebra of an extension of exact $C^\ast$-algebras is exact if the extension is locally split. The converse is also true, and follows from \cite{EffrosHaagerup-lifting}.
 If $\tau$ is nuclear then for any finite dimensional operator system $E \subseteq \mathfrak A$, there is a c.p.~lift $\tilde \tau \colon E \to \multialg{\mathfrak B}$ of $\tau|_E$ by the Choi--Effros lifting theorem \cite{ChoiEffros-lifting}.
 If $\iota \colon E \to \mathfrak A$ is the inclusion, then $(\iota,\tilde \tau) \colon E \to \mathfrak A \oplus_{\corona{\mathfrak B}} \multialg{\mathfrak B} \cong \mathfrak E$ is a c.p.~lift of $\iota$.
 Hence the extension is locally split and thus $\mathfrak E$ is exact.
 
 If $\mathfrak E$ is exact then it is locally split as noted above. Since $\mathfrak B$ is nuclear it follows from \cite{EffrosHaagerup-lifting} that for any separable $C^\ast$-subalgebra $\mathfrak A_0 \subseteq \mathfrak A$ there is a 
 c.p.~lift $\tilde \tau \colon \mathfrak A_0 \to \multialg{\mathfrak B}$ of $\tau|_{\mathfrak A_0}$. 
 Since $\mathfrak B$ is nuclear it follows that $\tilde \tau$ is weakly nuclear, and since $\mathfrak A_0$ is exact it follows from Proposition \ref{p:exactvsweaklynuc} that $\tilde \tau$ is nuclear.
 Hence $\tau|_{\mathfrak A_0} = \pi \circ \tilde \tau$ is nuclear. Since $\mathfrak A_0$ was arbitrarily chosen, it follows that $\tau$ is nuclear.
\end{proof}


\section{Ideal related completely positive selections}

In this section we prove ideal related selection results for completely positive maps, where we by ideal related mean $\mathsf X$-equivariant as defined below. The purpose of these selection results, is to construct ``many'' $\mathsf X$-equivariant c.p.~maps between two $\mathsf X$-$C^\ast$-algebras, which is important when one wishes to lift $\mathsf X$-equivariant c.p.~maps to $\mathsf X$-equivariant c.p.~maps.

\subsection{Actions of topological spaces on $C^\ast$-algebras}

When $\mathsf X$ is a topological space, we let $\mathbb O(\mathsf X)$ denote the complete lattice of open subsets of $\mathsf X$.
Also, for a $C^\ast$-algebra $\mathfrak A$, we let $\mathbb I(\mathfrak A)$ denote the complete lattice of two-sided, closed ideals in $\mathfrak A$.

\begin{definition}
Let $\mathsf X$ be a topological space. An \emph{action of $\mathsf X$ on a $C^\ast$-algebra $\mathfrak A$} is an order preserving map $\psi \colon \mathbb O(\mathsf X) \to \mathbb I(\mathfrak A)$,
i.e.~a map such that if $\mathsf U \subseteq \mathsf V$ in $\mathbb O(\mathsf X)$ then $\psi(\mathsf U) \subseteq \psi(\mathsf V)$.

A $C^\ast$-algebra $\mathfrak A$ together with an action $\psi$ of $\mathsf X$ on $\mathfrak A$, is called an \emph{$\mathsf X$-$C^\ast$-algebra}. It is customary to suppress the action $\psi$ in the notation, by simply saying that $\mathfrak A$ is an $\mathsf X$-$C^\ast$-algebra, and defining $\mathfrak A(\mathsf U) := \psi(\mathsf U)$ for $\mathsf U \in \mathbb O(\mathsf X)$.

A map $\phi \colon \mathfrak A \to \mathfrak B$ of $C^\ast$-algebras with actions of $\mathsf X$ is called \emph{$\mathsf X$-equivariant} if $\phi(\mathfrak A(\mathsf U)) \subseteq \mathfrak B(\mathsf U)$ for every $\mathsf U\in \mathbb O(\mathsf X)$.
\end{definition}

\begin{remark}\label{r:Xeqcone}
If $\mathsf X$ is a space acting on the $C^\ast$-algebras $\mathfrak A$ and $\mathfrak B$, then the set $CP(\mathsf X; \mathfrak A, \mathfrak B)$ of $\mathsf X$-equivariant c.p.~maps $\mathfrak A \to \mathfrak B$ is a closed operator convex cone.

In particular, the set of all nuclear, $\mathsf X$-equivariant c.p.~maps $\mathfrak A \to \mathfrak B$ is a closed operator convex cone, as this is the set $CP(\mathsf X; \mathfrak A, \mathfrak B) \cap CP_\nuc(\mathfrak A, \mathfrak B)$, and since being a closed operator convex cone is preserved under intersections.
\end{remark}

It is often necessary to impose stronger conditions on our actions.

\begin{definition}\label{d:actions}
Let $\mathfrak A$ be an $\mathsf X$-$C^\ast$-algebra. We say that $\mathfrak A$ is
\begin{itemize}
 \item \emph{finitely lower semicontinuous} if $\mathfrak A (\mathsf X) = \mathfrak A$, and if it respects finite infima, i.e.~for open subsets $\mathsf U$ and $\mathsf V$ of $\mathsf X$ we have 
 \[
 \mathfrak A(\mathsf U) \cap \mathfrak A(\mathsf V) = \mathfrak A(\mathsf U \cap \mathsf V),
 \]
 \item \emph{lower semicontinuous} if $\mathfrak A(\mathsf X) = \mathfrak A$, and if it respects arbitrary infima, i.e.~for any family $(\mathsf U_\lambda)$ of open subsets of $\mathsf X$ we have
 \[
 \bigcap_\lambda \mathfrak A(\mathsf U_\lambda) = \mathfrak A(\mathsf U),
 \]
 where $\mathsf U$ is the interior of $\bigcap_\lambda \mathsf U_\lambda$,
 \item \emph{finitely upper semicontinuous} if $\mathfrak A(\emptyset) = 0$, and if it respects finite suprema, i.e.~for open subsets $\mathsf U$ and $\mathsf V$ of $\mathsf X$ we have
 \[
  \mathfrak A(\mathsf U) + \mathfrak A(\mathsf V) = \mathfrak A(\mathsf U \cup \mathsf V),
 \]
 \item \emph{monotone upper semicontinuous} if it respects monotone suprema, i.e.~for any increasing net $(\mathsf U_\lambda)$ of open subsets of $\mathsf X$ we have
 \[
  \overline{ \bigcup_\lambda \mathfrak A(\mathsf U_\lambda)} = \mathfrak A(\bigcup_\lambda \mathsf U_\lambda),
 \]
 \item \emph{upper semicontinuous} if it is finitely and monotone upper semicontinuous,
\end{itemize}
\end{definition}

Note that an upper semicontinuous $\mathsf X$-$C^\ast$-algebra $\mathfrak C$ satisfies $\mathfrak C(\emptyset) = 0$. This condition ensures that the map $\Psi \colon \mathbb I(\mathfrak C) \to \mathbb O(\mathsf X)$ given by
\[
 \Psi(\mathfrak I) = \bigcup \{ \mathsf U \in \mathbb O(\mathsf X) : \mathfrak C(\mathsf U) \subseteq \mathfrak I\}
\]
is well-defined. This will be used in the proof of Proposition \ref{p:fromAtoC}. That a lower semicontinuous $\mathsf X$-$C^\ast$-algebra $\mathfrak C$ satisfies $\mathfrak C(\mathsf X)= \mathfrak C$, is for a similar reason.

\begin{definition}
Let $\mathfrak A$ be an $\mathsf X$-$C^\ast$-algebra, $a\in \mathfrak A$ and $\mathsf U\in \mathbb O(\mathsf X)$. We say that $a$ is \emph{$\mathsf U$-full}, if $\mathsf U$ is minimal amongst open sets $\mathsf V \in \mathbb O(\mathsf X)$ for which $a\in \mathfrak A(\mathsf V)$, i.e.~$a\in \mathfrak A(\mathsf U)$ and whenever 
$\mathsf V \in \mathbb O(\mathsf X)$ such that $a\in \mathfrak A(\mathsf V)$ then $\mathsf U \subseteq \mathsf V$.
\end{definition}

If $a$ is $\mathsf U$-full, then the set $\mathsf U$ is unique.

\begin{notation}
 Whenever $a \in \mathfrak A$ is $\mathsf U$-full for some $\mathsf U \in \mathbb O(\mathsf X)$, then we denote by $\mathsf U_a := \mathsf U$.
\end{notation}

Any element $a\in \mathfrak A$ in a $C^\ast$-algebra generates a two-sided closed ideal $\overline{\mathfrak A a \mathfrak A}$ which corresponds uniquely to an open subset $\mathsf U$ of $\Prim \mathfrak A$. If $\mathfrak A$ is equipped with the canonical action $\mathbb O(\Prim \mathfrak A) \xrightarrow{\cong} \mathbb I(\mathfrak A)$, then $a$ is $\mathsf U$-full for this set $\mathsf U\in \mathbb O(\Prim \mathfrak A)$, so $\mathfrak A(\mathsf U_a) = \overline{\mathfrak A a \mathfrak A}$.

If $\mathfrak A$ is a general $\mathsf X$-$C^\ast$-algebra, and $a\in \mathfrak A$ is $\mathsf U_a$-full, then one should think of $\mathsf U_a$ as being the open subset of $\mathsf X$ generated by $a$.

We will use the following result from \cite{GabeRuiz-absrep}. For the sake of completion, we give a proof.

\begin{proposition}
Let $\mathfrak A$ be an $\mathsf X$-$C^\ast$-algebra. 
Then $\mathfrak A$ is lower semicontinuous if and only if every element $a\in \mathfrak A$ is $\mathsf U_a$-full for some (unique) $\mathsf U_a\in \mathbb O(\mathsf X)$.
\end{proposition}
\begin{proof}
If $\mathfrak A$ is lower semicontinuous and $a\in \mathfrak A$, let $\mathsf U_a$ be the interior of the intersection of all open sets $\mathsf U\subseteq \mathsf X$ for which $a \in \mathfrak A(\mathsf U)$. As $\mathfrak A(\mathsf X) = \mathfrak A$, this construction is well-defined. By lower semicontinuity, $a\in \mathfrak A(\mathsf U_a)$, so $a$ is $\mathsf U_a$-full, as $\mathsf U_a$ is minimal amongst $\mathsf U\in \mathbb O(\mathsf X)$ for which $a\in \mathfrak A(\mathsf U)$.

Suppose every $a\in \mathfrak A$ is $\mathsf U_a$-full, let $(\mathsf U_\lambda)$ be a family of sets in $\mathbb O(\mathsf X)$, and $\mathsf U$ be the interior of the intersection of $(\mathsf U_\lambda)$. Clearly $\mathfrak A(\mathsf U) \subseteq \bigcap \mathfrak A(\mathsf U_\lambda)$. Let $a\in \bigcap \mathfrak A(\mathsf U_\lambda)$. As $a\in \mathfrak A(\mathsf U_\lambda)$ for all $\lambda$, it follows that $\mathsf U_a \subseteq \mathsf U_\lambda$ for all $\lambda$, and thus $\mathsf U_a \subseteq \mathsf U$. So $a\in \mathfrak A(\mathsf U)$ and thus $\mathfrak A(\mathsf U) = \bigcap \mathfrak A(\mathsf U_\lambda)$. Finally, suppose $a \in \mathfrak A \setminus \mathfrak A(\mathsf X)$. Then $a\in \mathfrak A(\mathsf U_a) \subseteq \mathfrak A(\mathsf X)$, a contradiction, so $\mathfrak A = \mathfrak A(\mathsf X)$.
\end{proof}

\begin{example}
 Let $\mathsf X$ be a locally compact Hausdorff space. A \emph{$C_0(\mathsf X)$-algebra} is a $C^\ast$-algebra $\mathfrak A$ together with an essential $\ast$-homomorphism $\Phi$ from $C_0(\mathsf X)$ into the center of $\multialg{\mathfrak A}$.
 Essential means that $\overline{\Phi(C_0(\mathsf X)) \mathfrak A} = \mathfrak A$.
 As described in \cite[Section 2.1]{MeyerNest-bootstrap} there is a one-to-one correspondence between such $\ast$-homomorphisms, and actions of $\mathsf X$ on $\mathfrak A$ which are finitely lower semicontinuous and upper semicontinuous.
 The induced action is given by $\mathfrak A(\mathsf U) = \overline{\mathfrak A \Phi(C_0(\mathsf U))}$ for $\mathsf U\in \mathbb O(\mathsf X)$.
 A $C_0(\mathsf X)$-algebra is called \emph{continuous} if for every $a\in \mathfrak A$, the set $\mathsf U_a := \{ x\in \mathsf X : \| a + \mathfrak A(\mathsf X \setminus \{ x\})\| > 0\}$ is open.
 If this is the case it is easily seen that $a$ is $\mathsf U_a$-full, and conversely, if $a$ is $\mathsf U_a$-full, then $\mathsf U_a = \{ x \in \mathsf X : \| a + \mathfrak A(\mathsf X \setminus \{ x\})\| > 0\}$.
 Thus $\mathfrak A$ as an $\mathsf X$-$C^\ast$-algebra is continuous if and only if $\mathfrak A$ as a $C_0(\mathsf X)$-algebra is continuous.
 \end{example}

\begin{observation}
 Let $\mathfrak A$ and $\mathfrak B$ be $\mathsf X$-$C^\ast$-algebras with $\mathfrak A$ lower semicontinuous. 
 Then a map $\phi \colon \mathfrak A \to \mathfrak B$ is $\mathsf X$-equivariant if and only if for all (positive) $a\in \mathfrak A$, $\phi(a) \in \mathfrak B(\mathsf U_a)$.
\end{observation}

When $\mathfrak B$ is an $\mathsf X$-$C^\ast$-algebra and $\mathfrak D$ is any $C^\ast$-algebra, then the spatial tensor product $\mathfrak B \otimes \mathfrak D$ is canonically an $\mathsf X$-$C^\ast$-algebra by the action
$\mathsf U \mapsto \mathfrak B(\mathsf U) \otimes \mathfrak D$.

\begin{lemma}\label{l:actionstensorD}
Let $\mathfrak B$ be an $\mathsf X$-$C^\ast$-algebra and $\mathfrak D$ be a $C^\ast$-algebra. Whenever $\mathfrak B$ is monotone (resp.~finitely) upper semicontinuous, then so is $\mathfrak B \otimes \mathfrak D$.

Suppose, moreover, that $\mathfrak B$ or $\mathfrak D$ is exact. If $\mathfrak B$ is (finitely) lower semicontinuous, then so is $\mathfrak B \otimes \mathfrak D$.
\end{lemma}
\begin{proof}
Monotone upper semicontinuity: this is clearly preserved when tensoring with $\mathfrak D$. 

Finite upper semicontinuity: Clearly $(\mathfrak B \otimes \mathfrak D)(\emptyset) = 0$.
If $\mathfrak I$ and $\mathfrak J$ are two-sided, closed ideals in $\mathfrak B$, then $(\mathfrak I + \mathfrak J) \otimes \mathfrak D$ is the closed linear span of elementary tensors. 
Since any element in $\mathfrak I +\mathfrak J$ can be written as $x+y$ with $x\in \mathfrak I$ and $y \in \mathfrak J$ it easily follows that 
$(\mathfrak I + \mathfrak J)\otimes \mathfrak D = \mathfrak I \otimes \mathfrak D + \mathfrak J \otimes \mathfrak D$. Thus finite upper semicontinuity is preserved when tensoring with $\mathfrak D$.

(Finite) lower semicontinuity:  Clearly $(\mathfrak B \otimes \mathfrak D)(\mathsf X) = \mathfrak B \otimes \mathfrak D$. Let $(\mathfrak J_\lambda)$ be a family of two-sided, closed ideals in $\mathfrak B$, let $\mathfrak J = \bigcap \mathfrak J_\lambda$, and let $\mathfrak I = \bigcap (\mathfrak J_\lambda \otimes \mathfrak D)$. 
Clearly $\mathfrak J \otimes \mathfrak D \subseteq \mathfrak I$. 
By \cite[Corollary 9.4.6]{BrownOzawa-book-approx}, $\mathfrak I$ is the closed linear span of all elementary tensors $b\otimes d$ with $b\in \mathfrak B$, $d\in \mathfrak D$ and $b\otimes d \in \mathfrak I$. 
For such $b,d$ it easily follows that $b\in \mathfrak J$, so $\mathfrak I = \mathfrak J \otimes \mathfrak D$. It clearly follows that (finite) lower semicontinuity is preserved when tensoring with $\mathfrak D$.
\end{proof}


\subsection{Selection results}

We will be applying a variation of one of the remarkable selection theorems of Michael \cite{Michael-selection}. To do this we need some notation. Let $\mathsf Y$ and $\mathsf Z$ be topological spaces. 
A \emph{carrier} from $\mathsf Y$ to $\mathsf Z$ is a map $\Gamma \colon \mathsf Y \to 2^{\mathsf Z}$, where $2^{\mathsf Z}$ is the set of non-empty subsets of $\mathsf Z$.
We say that $\Gamma$ is \emph{lower semicontinuous} if for every open subset $\mathsf U$ of $\mathsf Z$, the set
\[
 \{ y \in \mathsf Y : \Gamma(y) \cap \mathsf U \neq \emptyset\}
\]
is open in $\mathsf Y$. 
One of Michael's selection theorems \cite[Theorem 1.2]{Michael-aselectionthm} implies that if $\mathsf Y$ is a paracompact $T_1$-space (e.g.~a second countable, locally compact Hausdorff space), 
if $(\mathsf Z^\ast)_1$ is the unit ball of the dual space of a separable Banach space $\mathsf Z$, and if $\Gamma$ is a lower semicontinuous carrier from $\mathsf Y$ to $(\mathsf Z^\ast)_1$
such that $\Gamma(y)$ is a weak$^\ast$-closed convex set in $(\mathsf Z^\ast)_1$ for all $y \in \mathsf Y$, then there exists a continuous map $\gamma \colon \mathsf Y \to (\mathsf Z^\ast)_1$ such that $\gamma(y) \in \Gamma(y)$ for all $y$.

When $\mathsf Y$ is a locally compact Hausdorff space, $y\in \mathsf Y$ we let $\mathrm{ev}_y \colon C_0(\mathsf Y) \to \mathbb C$ denote the $\ast$-homomorphism which is evaluation in $y$.

We will use the following ideal related selection result. A very similar result can be found in the preprint \cite[Lemma A.15]{HarnischKirchberg-primitive}.\footnote{In \cite[Lemma A.15]{HarnischKirchberg-primitive} they assume that
$\mathfrak A$ is separable and $\mathsf Y$ is any locally compact Hausdorff space. In their proof they use an unspecified selection theorem of Michael from \cite{Michael-selection}. 
The selection theorem with weakest preliminary conditions in \cite{Michael-selection} requires $\mathsf Y$ to be normal. However, there are examples of locally compact Hausdorff spaces (not second countable) which are not normal.
Thus the proof of \cite[Lemma A.15]{HarnischKirchberg-primitive} requires more arguments than are given, if one should apply the selection theorems of Michael.}

\begin{lemma}\label{l:cpselection}
 Let $\mathsf Y$ be a second countable, locally compact Hausdorff space, and let $\mathfrak A$ be a lower semicontinuous $\mathsf Y$-$C^\ast$-algebra.
 For any distinct points $y_1,\dots,y_n\in\mathsf Y$, and any quasi-states $\eta_k$ on $\mathfrak A/\mathfrak A(\mathsf Y \setminus \{ y_k\})$, there is a contractive $\mathsf Y$-equivariant c.p.~map $\phi \colon \mathfrak A \to C_0(\mathsf Y)$,
 such that
 \[
  \mathrm{ev}_{y_k}(\phi(a)) = \eta_k(a + \mathfrak A(\mathsf Y\setminus \{ y_k\}))
 \]
 for all $a\in \mathfrak A$.
\end{lemma}
\begin{proof}
 First note that the forced unitisation $\mathfrak A^\dagger$ has a canonical lower semicontinuous action of $\mathsf Y$ given by $\mathfrak A^{\dagger}(\mathsf V) = \mathfrak A(\mathsf V)$ when 
 $\mathsf V \neq \mathsf Y$ and $\mathfrak A^\dagger(\mathsf Y) = \mathfrak A^\dagger$.
 Let $P(\mathfrak A^\dagger)$ be the space of pure states on $\mathfrak A^\dagger$.
 Let $\Gamma \colon \mathsf Y \to 2^{P(\mathfrak A^\dagger)}$ be the carrier given by 
 \[
  \Gamma(y) = \{ \eta \in P(\mathfrak A^\dagger) : \eta(\mathfrak A^\dagger(\mathsf Y \setminus\{y\})) = 0 \} = \{ \eta \in P(\mathfrak A^\dagger) : \eta(\mathfrak A(\mathsf Y \setminus\{y\})) =  0  \}.
 \]
 We claim that $\Gamma$ is lower semicontinuous. To see this, first recall (e.g.~\cite[Theorem 4.3.3]{Pedersen-book-automorphism}) that the continuous map $P(\mathfrak A^\dagger) \to \Prim \mathfrak A^\dagger$ given by $\eta \mapsto \ker \pi_\eta$,
 where $\pi_\eta$ is the GNS representation, is an open map. Thus this induces a map $\mathbb O(P(\mathfrak A^\dagger)) \to \mathbb O(\Prim \mathfrak A^\dagger) \cong \mathbb I(\mathfrak A^\dagger)$.
 We may construct a map $\mathbb I(\mathfrak A^\dagger) \to \mathbb O(\mathsf Y)$, by
 \[
  \mathfrak J \mapsto ( \bigcap_{\mathsf V \in \mathbb O(\mathsf Y), \mathfrak J \subseteq \mathfrak A^{\dagger}(\mathsf V)} \mathsf V)^\circ.
 \]
 Since $\mathfrak A^\dagger$ is lower semicontinuous, $\mathfrak J$ is mapped to the unique smallest open subset $\mathsf V$ of $\mathsf Y$ for which $\mathfrak J \subseteq \mathfrak A^{\dagger}(\mathsf V)$. Let $\Phi$ denote the composition
 \[
  \mathbb O(P(\mathfrak A^\dagger)) \to \mathbb O(\Prim \mathfrak A^\dagger) \cong \mathbb I(\mathfrak A^\dagger) \to \mathbb O(\mathsf Y).
 \]
 The map $\Phi$ can be described as follows: Let $\mathsf U\in \mathbb O(P(\mathfrak A^\dagger))$. 
 Then there is a unique two-sided, closed ideal $\mathfrak J_\mathsf{U}$ in $\mathfrak A^{\dagger}$, such that
 \[
  \{ \ker \pi_\eta : \eta \in \mathsf U\} = \{ \mathfrak p \in \Prim \mathfrak A^{\dagger} : \mathfrak J_\mathsf{U} \not \subseteq \mathfrak p\},
 \]
 and $\Phi(\mathsf U)$ is the unique smallest open subset of $\mathsf Y$ such that $\mathfrak J_\mathsf{U} \subseteq \mathfrak A^{\dagger}(\Phi(\mathsf U))$.
 
 We claim that for any $\mathsf U\in \mathbb O(P(\mathfrak A^\dagger))$ we have
 \[
  \{ y \in \mathsf Y : \Gamma(y) \cap \mathsf U \neq \emptyset \} = \Phi(\mathsf U)
 \]
 and thus $\Gamma$ is a lower semicontinuous carrier. That this is true follows from the following:
 \begin{eqnarray*}
  \Gamma(y) \cap \mathsf U = \emptyset &\Leftrightarrow& \text{for every } \eta \in \mathsf U \text{ we have } \mathfrak A^\dagger(\mathsf Y \setminus \{y\}) \not \subseteq \ker \pi_{\eta} \\
  &\Leftrightarrow& \{ \mathfrak p \in \Prim \mathfrak A^\dagger : \mathfrak J_\mathsf{U} \not \subseteq \mathfrak p \} \subseteq \{ \mathfrak p \in \Prim \mathfrak A^\dagger : \mathfrak A^\dagger(\mathsf Y\setminus \{y\}) \not \subseteq \mathfrak p \} \\
  &\Leftrightarrow& \mathfrak J_\mathsf{U} \subseteq \mathfrak A^\dagger(\mathsf Y \setminus\{ y \}) \\
  &\Leftrightarrow& \Phi(\mathsf U) \subseteq \mathsf Y \setminus \{y\} \\
  &\Leftrightarrow& y \notin \Phi(\mathsf U).
 \end{eqnarray*}
 
 Let $\mathscr Q(\mathfrak A) \subseteq \mathfrak A^\ast$ denote the quasi-state space of $\mathfrak A$ and
 \[
  K_y := \{ \eta \in \mathscr Q(\mathfrak A) : \eta(\mathfrak A(\mathsf Y \setminus \{y\})) = 0\}
 \]
 for every $y\in \mathsf Y$. Recall, that the restriction map $(\mathfrak A^\dagger)^\ast \to \mathfrak A^\ast$ induces a homeomorphism $P(\mathfrak A^\dagger) \to \{ 0 \} \cup P(\mathfrak A)$. 
 Moreover, under this identification, the closed convex hull of $\Gamma(y)$ is exactly $K_y$.
 Thus it follows from \cite[Propositions 2.3 and 2.6]{Michael-selection} that the carrier $\Gamma_1 \colon \mathsf Y \to 2^{(\mathfrak A^\ast)_1}$ given by $\Gamma_1(y) = K_y$, is lower semicontinuous.
 
 Let $\mathsf A = \{ y_1, \dots, y_n\} \subseteq \mathsf Y$ which is (obviously) a closed subspace, and let $\pi_y \colon \mathfrak A \to \mathfrak A/\mathfrak A(\mathsf Y\setminus \{y\})$ be the quotient map for each $y\in \mathsf Y$. 
 The map $g_0 \colon \mathsf A \to \mathscr Q(\mathfrak A)$ given by $g_0(y_k) = \eta_k \circ \pi_{y_k}$ is clearly continuous and $g_0(y_k) \in K_{y_k}= \Gamma_1(y_k)$.
 Thus it follows from \cite[Example 1.3*]{Michael-selection} that the carrier $\Gamma_2 \colon \mathsf Y \to 2^{(\mathfrak A^\ast)_1}$ given by
 \[
  \Gamma_2(y) = \left\{ \begin{array}{ll}
                         \{ g_0(y) \}, & \text{ if } y \in \mathsf A\\
                         K_y, & \text{ otherwise}
                        \end{array}
                \right.
 \]
 is lower semicontinuous. Since $\mathsf Y$ is a paracompact $T_1$-space, and $\Gamma_2(y)$ is a closed convex space for every $y\in \mathsf Y$, 
 it follows from \cite[Theorem 1.2]{Michael-aselectionthm} that there exists a continuous map $g \colon \mathsf Y \to (\mathfrak A^\ast)_1$, such that $g(y) \in \Gamma_2(y)$ for all $y\in \mathsf Y$.
 
 Now, let $\hat \phi \colon \mathfrak A \to C_b(\mathsf Y)$ be given by $\mathrm{ev}_y \circ \hat \phi(a) = g(y)(a)$. Since $\mathrm{ev}_y \circ \hat \phi$ is a contractive c.p.~map (a quasi-state) for every $y\in \mathsf Y$, it follows that $\hat \phi$ is a contractive c.p.~map. 
 Pick a positive contraction $f$ in $C_0(\mathsf Y)$ such that $f(y_k) = 1$ for $k=1,\dots, n$. Then $\phi \colon \mathfrak A \to C_0(\mathsf Y)$ given by $\phi(a) = f \cdot \hat \phi(a)$ is again a contractive c.p.~map.
 Moreover, we clearly have
 \[
  \mathrm{ev}_{y_k}(\phi(a)) = f(y_k) \cdot g(y_k)(a) = \eta_k(a + \mathfrak A(\mathsf Y \setminus \{ y_k\})).
 \]
 Thus it remains to show that $\phi$ is $\mathsf Y$-equivariant. 
 
 Let $\mathsf V \in \mathbb O(\mathsf Y)$ and $a\in \mathfrak A(\mathsf V)$. For every $y \notin \mathsf V$ we have $\mathsf V \subseteq \mathsf Y \setminus \{ y\}$ and thus $a\in \mathfrak A(\mathsf Y \setminus \{y\})$.
 Since $\mathrm{ev}_y \circ \phi(a) \in K_y$ it follows that $\mathrm{ev}_y \circ \phi(a) = 0$, and thus $\phi(a) \in C_0(\mathsf Y \setminus \{y\})$. Hence we have
 \[
  \phi(a) \in \bigcap_{y \notin \mathsf V} C_0(\mathsf Y \setminus \{ y\} ) = C_0(\mathsf V),
 \]
 which implies that $\phi$ is $\mathsf Y$-equivariant.
\end{proof}

The above lemma lets us prove the following selection result for $\mathsf X$-equivariant maps.
Recall, that when $\mathfrak A$ is a lower semicontinuous $\mathsf X$-$C^\ast$-algebra and $a\in \mathfrak A$, then $\mathsf U_a$ denotes the unique smallest open subset of $\mathsf X$ for which $a\in \mathfrak A(\mathsf U_a)$.

\begin{proposition}\label{p:fromAtoC}
Let $\mathfrak A$ be a lower semicontinuous $\mathsf X$-$C^\ast$-algebra and let $\mathfrak C$ be a separable, commutative, upper semicontinuous $\mathsf X$-$C^\ast$-algebra. For every positive $a\in \mathfrak A$ there exists an $\mathsf X$-equivariant c.p.~map $\phi \colon \mathfrak A \to \mathfrak C$ such that $\phi(a)$ is strictly positive in $\mathfrak C(\mathsf U_a)$.
\end{proposition}
\begin{proof}
Let $\mathsf Y = \Prim \mathfrak C$ such that $\mathfrak C = C_0(\mathsf Y)$. 
To avoid confusion, we will write $\mathfrak C$ when we are using the $\mathsf X$-$C^\ast$-algebra structure, and write $C_0(\mathsf Y)$ when we consider $\mathfrak C = C_0(\mathsf Y)$ with the tight $\mathsf Y$-$C^\ast$-algebra structure.
The idea of the proof, is to construct a lower semicontinuous action $\tilde \Psi$ of $\mathsf Y$ on $\mathfrak A$, such that a c.p.~map $\mathfrak A \to \mathfrak C$ is $\mathsf X$-equivariant if and only if the same map $(\mathfrak A, \tilde \Psi) \to C_0(\mathsf Y)$ is $\mathsf Y$ equivariant. When this is done we can apply Lemma \ref{l:cpselection} to construct $\mathsf X$-equivariant c.p.~maps $\mathfrak A \to \mathfrak C$.

Construct a map $\Psi \colon \mathbb O(\mathsf Y) \to \mathbb O(\mathsf X)$ given by
 \[
  \Psi(\mathsf V) = \bigcup \{ \mathsf U \in \mathbb O(\mathsf X) : \mathfrak C(\mathsf U) \subseteq C_0(\mathsf V)\}.
 \]
 Since the action of $\mathsf X$ on $\mathfrak C$ is upper semicontinuous, $\Psi(\mathsf V)$ is the \emph{unique largest} open subset of $\mathsf X$ such that $\mathfrak C(\Psi(\mathsf V)) \subseteq C_0(\mathsf V)$, 
 in the sense that $\mathfrak C(\Psi(\mathsf V)) \subseteq C_0(\mathsf V)$ and if $\mathsf U \in \mathbb O(\mathsf X)$ satisfies $\mathfrak C(\mathsf U) \subseteq C_0(\mathsf V)$ then $\mathsf U \subseteq \Psi(\mathsf V)$.
 We clearly have that $\Psi$ is order preserving and that $\Psi(\mathsf Y) = \mathsf X$. 
 We want to show that whenever $(\mathsf V_\alpha)$ is a family of open subsets of $\mathsf Y$, and $\mathsf V$ is the interior of $\bigcap \mathsf V_\alpha$, then $\Psi(\mathsf V)$ is the interior of $\bigcap \Psi(\mathsf V_\alpha)$.
 For now, we let $\mathsf W$ denote the interior of $\bigcap \Psi(\mathsf V_\alpha)$.
 
 Since $\Psi$ is order preserving we clearly have that $\Psi(\mathsf V) \subseteq \mathsf W$. For the converse inclusion we have that $\mathfrak C(\mathsf W) \subseteq \mathfrak C(\Psi(\mathsf V_\alpha)) \subseteq C_0(\mathsf V_\alpha)$ for each $\alpha$.
 Hence $\mathfrak C(\mathsf W) \subseteq \bigcap C_0(\mathsf V_\alpha) = C_0(\mathsf V)$. It follows from the definition of $\Psi$ that $\mathsf W \subseteq \Psi(\mathsf V)$, and thus we have equality.
 
 Let $\tilde \Psi\colon \mathbb O(\mathsf Y) \to \mathbb I(\mathfrak A)$ be the action of $\mathsf Y$ on $\mathfrak A$ given by $\tilde \Psi(\mathsf V) = \mathfrak A(\Psi(\mathsf V))$. 
 It follows that $\tilde \Psi(\mathsf Y) = \mathfrak A$, and since the action of $\mathsf X$ on $\mathfrak A$ is lower semicontinuous, so is the action $\tilde \Psi$, by what we have proven above. 
 Thus $(\mathfrak A, \tilde \Psi)$ is a lower semicontinuous $\mathsf Y$-$C^\ast$-algebra.
 
 We will prove that $CP(\mathsf X; \mathfrak A, \mathfrak C) = CP(\mathsf Y; (\mathfrak A, \tilde \Psi), C_0(\mathsf Y))$. To see this, first note that $\mathfrak C(\Psi(\mathsf V)) \subseteq C_0(\mathsf V)$ for all $\mathsf V \in \mathbb O(\mathsf Y)$.
Thus if $\phi$ is $\mathsf X$-equivariant then
\[
 \phi ( \tilde \Psi(\mathsf V)) = \phi(\mathfrak A(\Psi(\mathsf V))) \subseteq \mathfrak C(\Psi(\mathsf V)) \subseteq C_0(\mathsf V).
\]
and thus $\phi$ is $\mathsf Y$-equivariant.
For $\mathsf U \in \mathbb O(\mathsf X)$ let $\mathsf V^{\mathsf U}\in \mathbb O(\mathsf Y)$ be such that $\mathfrak C(\mathsf U) = C_0(\mathsf V^\mathsf{U})$. Since $\Psi(\mathsf V^\mathsf{U})$ is the unique largest open subset of $\mathsf X$ such that $\mathfrak C(\Psi(\mathsf V^{\mathsf U})) \subseteq C_0(\mathsf V^\mathsf{U}) = \mathfrak C(\mathsf U)$ it follows that $\mathsf U \subseteq \Psi(\mathsf V^{\mathsf U})$.
Thus if $\psi$ is $\mathsf Y$-equivariant then
\[
 \psi(\mathfrak A(\mathsf U)) \subseteq \psi(\mathfrak A(\Psi(\mathsf V^{\mathsf U}))) = \psi(\tilde \Psi(\mathsf V^{\mathsf U})) \subseteq C_0(\mathsf V^{\mathsf U}) = \mathfrak C(\mathsf U).
\]
Hence it follows that $CP(\mathsf X;\mathfrak A, \mathfrak C) = CP(\mathsf Y; (\mathfrak A, \tilde \Psi), C_0(\mathsf Y))$.

Fix a positive $a\in \mathfrak A$.
Recall that $\mathsf U_a$ is the open subset of $\mathsf X$ such that $a$ is $\mathsf U_a$-full, when considering $\mathfrak A$ with the $\mathsf X$-$C^\ast$-algebra structure. 
Since $(\mathfrak A, \tilde \Psi)$ is a lower semicontinuous $\mathsf Y$-$C^\ast$-algebra, we may find a unique open subset $\mathsf V_a$ of $\mathsf Y$ such that $a$ is $\mathsf V_a$-full when considered with the $\mathsf Y$-$C^\ast$-algebra structure.
We will show that $\mathfrak C(\mathsf U_a) = C_0(\mathsf V_a)$.

Since $a \in \tilde \Psi(\mathsf V_a) = \mathfrak A(\Psi(\mathsf V_a))$ it follows from $\mathsf U_a$-fullness that $\mathsf U_a \subseteq \Psi(\mathsf V_a)$ and thus $\mathfrak C(\mathsf U_a) \subseteq \mathfrak C(\Psi(\mathsf V_a)) \subseteq C_0(\mathsf V_a)$.
Let $\mathsf W \in \mathbb O(\mathsf Y)$ be such that $C_0(\mathsf W) = \mathfrak C(\mathsf U_a)$. Then $\mathsf U_a \subseteq \Psi(\mathsf W)$ by the definition of $\Psi$.
This implies that $a \in \mathfrak A(\mathsf U_a) \subseteq \mathfrak A(\Psi(\mathsf W)) = \tilde \Psi(\mathsf W)$. By $\mathsf V_a$-fullness it follows that $\mathsf V_a \subseteq \mathsf W$ and thus $C_0(\mathsf V_a) \subseteq C_0(\mathsf W) = \mathfrak C(\mathsf U_a)$.
This shows that $\mathfrak C(\mathsf U_a) = C_0(\mathsf V_a)$.

Our goal is to construct an $\mathsf X$-equivariant c.p.~map $\psi \colon \mathfrak A \to \mathfrak C$ such that $\psi(a)$ is strictly positive in $\mathfrak C(\mathsf U_a)$. 
Equivalently, by what we have shown above, we should construct a $\mathsf Y$-equivariant c.p.~map $\psi \colon (\mathfrak A, \tilde \Psi) \to C_0(\mathsf Y)$ such that $\psi(a)$ is strictly positive in $C_0(\mathsf V_a)$.

Suppose that $\mathsf V_a = \emptyset$. Then $C_0(\mathsf V_a) = 0$, and thus letting $\psi$ be the zero map will suffice. Thus suppose that $\mathsf V_a \neq \emptyset$.
For each $y\in \mathsf V_a$ we have that $\|a  + \mathfrak A(\mathsf Y \setminus \{ y\})\| > 0$. In fact, if $a \in \mathfrak A(\mathsf Y \setminus \{y\})$ then we would have
\[
 a \in \mathfrak A(\mathsf V_a) \cap \mathfrak A(\mathsf Y\setminus \{ y\}) = \mathfrak A(\mathsf V_a \setminus \{y\})
\]
which contradicts that $a$ is $\mathsf V_a$-full. Let $\eta_y$ be a state on $\mathfrak A/\mathfrak A(\mathsf Y \setminus \{ y \})$ such that $\eta_y(a + \mathfrak A(\mathsf Y \setminus \{ y\})) = \| a + \mathfrak A(\mathsf Y \setminus \{ y\})\|$.
By Lemma \ref{l:cpselection} there is a contractive $\mathsf Y$-equivariant c.p.~map $\psi_y \colon (\mathfrak A, \tilde \Psi) \to C_0(\mathsf Y)$ such that $\mathrm{ev}_y \circ \psi_y(a) = \eta_y(a + \mathfrak A(\mathsf Y \setminus \{ y\})) > 0$. Let $\mathsf W_y \subseteq \mathsf V_a$ be an open neighbourhood of $y$ such that $\mathrm{ev}_z \circ \psi_y(a) >0$ for all $z\in \mathsf W_y$. Then $(\mathsf W_y)_{y\in \mathsf V_a}$ is an open cover of $\mathsf V_a$. 
Since $\mathsf Y$ is second countable (as $\mathfrak C$ is separable) $\mathsf V_a$ is $\sigma$-compact, so we may find a sequence $(y_n)$ of points in $\mathsf V_a$ such that $(\mathsf W_{y_n})_{n\in \mathbb N}$ covers $\mathsf V_a$.
Let $\psi = \sum_{n=1}^\infty 2^{-n} \psi_{y_n}$ which is clearly a contractive $\mathsf Y$-equivariant c.p.~map. Clearly $0 < \mathrm{ev}_y(\psi(a))$ for every $y\in \mathsf V_a$. 
Since $\psi(a) \in C_0(\mathsf V_a)$ by $\mathsf Y$-equivariance, it follows that $\psi(a)$ is strictly positive in $C_0(\mathsf V_a)$.
\end{proof}


\subsection{Property (UBS)}

\begin{definition}\label{d:propertyubs}
 Let $\mathfrak B$ be an $\mathsf X$-$C^\ast$-algebra. 
 If $\mathfrak C$ is a separable, commutative, upper semicontinuous $\mathsf X$-$C^\ast$-algebra, we will say that $\mathfrak B$ has 
\emph{Property (UBS) with respect to $\mathfrak C$} if there exists a c.p.~map $\Phi\colon \mathfrak C \to \multialg{\mathfrak B}$ such that
 $\mathfrak B(\mathsf U) = \overline{\mathfrak B \Phi(\mathfrak C(\mathsf U)) \mathfrak B}$ for all $\mathsf U \in \mathbb O(\mathsf X)$.

We will say that $\mathfrak B$ has \emph{Property (UBS)} if it has Property (UBS) with respect to $\mathfrak C$ for some separable, commutative, upper semicontinuous $\mathsf X$-$C^\ast$-algebra $\mathfrak C$.
\end{definition}

\begin{remark}\label{r:ubsfactorB}
If $\mathfrak B$ in the above definition is $\sigma$-unital, then we may assume that the c.p.~map $\Phi$ factors through $\mathfrak B$. In fact, one may simply replace $\Phi$ in the above definition with $b\Phi(-)b$ for some strictly positive element $b\in \mathfrak B$.
\end{remark}

The name (UBS) has been chosen, since these $\mathsf X$-$C^\ast$-algebras resemble the \underline{u}pper semicontinuous $C^\ast$-\underline{b}undles over a \underline{s}econd countable, locally compact Hausdorff space,
as seen in the following example.

\begin{example}
 Let $\mathsf X$ be a second countable, locally compact Hausdorff space. It was shown in \cite{Nilsen-bundles} that any upper semicontinuous $C^\ast$-bundle over $\mathsf X$, may be considered, in a natural way, as a $C_0(\mathsf X)$-algebra,
 i.e.~as a $C^\ast$-algebra $\mathfrak B$ together with an essential $\ast$-homomorphism $\Phi \colon C_0(\mathsf X) \to \mathcal Z \multialg{\mathfrak B}$, where $\mathcal Z \multialg{\mathfrak B}$ is the centre of the multiplier algebra.
 This induces an action of $\mathsf X$ on $\mathfrak B$ given by 
\[
\mathfrak B(\mathsf U) = \overline{\mathfrak B \Phi( C_0(\mathsf U))} = \overline{\mathfrak B \Phi( C_0(\mathsf U))\mathfrak B}.
\]
Thus $\mathfrak B$ with this action has Property (UBS) with respect to $C_0(\mathsf X)$.
\end{example}

\begin{example}
 Let $\mathsf X$ be a finite space, and $\mathfrak B$ be an upper semicontinuous $\mathsf X$-$C^\ast$-algebra such that $\mathfrak B(\mathsf U)$ is $\sigma$-unital for each $\mathsf U \in \mathbb O(\mathsf X)$.
 Then $\mathfrak B$ has Property (UBS). Such $\mathsf X$-$C^\ast$-algebras are considered in \cite{GabeRuiz-absrep}.
 
 To see that $\mathfrak B$ has Property (UBS), let $\mathfrak C = \bigoplus_{x\in \mathsf X} \mathbb C$ with the action of $\mathsf X$ given by $\mathfrak C(\mathsf U) = \bigoplus_{x\in \mathsf U} \mathbb C$ for $\mathsf U\in \mathbb O(\mathsf X)$.
 This is easily seen to be an upper semicontinuous $\mathsf X$-$C^\ast$-algebra.
 For $x\in \mathsf X$ let $\mathsf U^x$ be the smallest open subset of $\mathsf X$ containing $x$ and let $h_x$ be a strictly positive element in $\mathfrak B(\mathsf U^x)$.
 The c.p.~map $\Phi \colon \mathfrak C \to \mathfrak B$, which maps $1$ in the coordinate corresponding to $x$ to $h_x$, satisfies the condition in Definition \ref{d:propertyubs}.
\end{example}

The following is the reason that we are interested in Property (UBS).

\begin{proposition}\label{p:Xfullmaps}
Let $\mathfrak A$ be a lower semicontinuous $\mathsf X$-$C^\ast$-algebra and $\mathfrak B$ be a $\sigma$-unital $\mathsf X$-$C^\ast$-algebra with Property (UBS). For any positive $a\in \mathfrak A$, there exists a nuclear, $\mathsf X$-equivariant c.p.~map $\phi\colon \mathfrak A \to \mathfrak B$ such that $\overline{\mathfrak B \phi(a) \mathfrak B} = \mathfrak B(\mathsf U_a)$.
\end{proposition}
\begin{proof}
As $\mathfrak B$ has Property (UBS), we may find a separable, commutative, upper semicontinuous $\mathsf X$-$C^\ast$-algebra $\mathfrak C$, and a c.p.~map $\Phi \colon \mathfrak C \to \mathfrak B$ such that $\mathfrak B(\mathsf U) = \overline{\mathfrak B \Phi(\mathfrak C(\mathsf U)) \mathfrak B}$ for all $\mathsf U\in \mathbb O(\mathsf X)$. Clearly $\Phi$ is $\mathsf X$-equivariant.

Fix $a\in \mathfrak A$ positive. By Proposition \ref{p:fromAtoC}, there is an $\mathsf X$-equivariant c.p.~map $\psi \colon \mathfrak A \to \mathfrak C$ such that $\psi(a)$ is strictly positive in $\mathfrak C(\mathsf U_a)$. Let $\phi = \Phi \circ \psi$, which is $\mathsf X$-equivariant as both $\psi$ and $\Phi$ are, and nuclear since it factors through a commutative $C^\ast$-algebra. Also, 
\[
\overline{\mathfrak B \phi(a) \mathfrak B} = \overline{\mathfrak B \Phi(\mathfrak C(\mathsf U_a)) \mathfrak B} = \mathfrak B(\mathsf U_a). \qedhere
\]
\end{proof}

To give (many) more examples of $\mathsf X$-$C^\ast$-algebras with Property (UBS), we will use the following lemma. Recall, that we let $\otimes$ denote the spatial tensor product.

\begin{lemma}\label{l:coolstate}
 Let $\mathfrak D$ be a separable, exact $C^\ast$-algebra. Then there exists a state $\eta$ on $\mathfrak D$ with the following property: for any $C^\ast$-algebra $\mathfrak B$ and any two-sided, closed ideal $\mathfrak J$ in $\mathfrak B$,
 it holds for any $x\in \mathfrak B \otimes \mathfrak D$ that $x \in \mathfrak J \otimes \mathfrak D$ if and only if $(id \otimes \eta)(x^\ast x) \in \mathfrak J$.
\end{lemma}
\begin{proof}
 Let $(\eta_n)$ be a weak-$\ast$ dense sequence in the state space of $\mathfrak D$, and let $\eta = \sum_{n=1}^\infty 2^{-n} \eta_n$.
 Let $\mathfrak J, \mathfrak B$ and $x$ be given. 
 By \cite[Corollary IV.3.4.2]{Blackadar-book-opalg}, $x \in \mathfrak J \otimes \mathfrak D$ if and only if $(id \otimes \eta')(x^\ast x) \in \mathfrak J$ for every state $\eta'$ on $\mathfrak D$.
 Clearly it suffices to only consider the case where $\eta'$ runs through all $\eta_n$ since these sit densely in the state space. 
 But since $(id \otimes \eta_n)(x^\ast x)$ is positive for each $n$, and $\mathfrak J$ is a hereditary $C^\ast$-subalgebra of $\mathfrak B$,
 it follows that $(id \otimes \eta_n)(x^\ast x) \in \mathfrak J$ for all $n$ if and only if $\sum_{n=1}^\infty 2^{-n} (id \otimes \eta_n)(x^\ast x) = (id \otimes \eta)(x^\ast x) \in \mathfrak J$.
\end{proof}

\begin{proposition}\label{p:ubstensorD}
Let $\mathfrak B$ be a $\sigma$-unital $\mathsf X$-$C^\ast$-algebra, let $\mathfrak D$ be a separable, exact $C^\ast$-algebra, and let $\mathfrak C$ be a separable, commutative, upper semicontinuous $\mathsf X$-$ C^\ast$-algebra. Then $\mathfrak B$ has Property (UBS) with respect to $\mathfrak C$ if and only if $\mathfrak B \otimes \mathfrak D$ has Property (UBS) with respect to $\mathfrak C$.
\end{proposition}
\begin{proof}
If $\mathfrak B$ has Property (UBS) with respect to $\mathfrak C$, and $\tilde \Phi\colon \mathfrak C \to \multialg{\mathfrak B}$ is a c.p.~map as in Definition \ref{d:propertyubs}, then 
\[
\Phi = \tilde \Phi \otimes 1_{\multialg{\mathfrak D}} \colon \mathfrak C \to \multialg{\mathfrak B} \otimes \multialg{\mathfrak D} \hookrightarrow \multialg{\mathfrak B \otimes \mathfrak D},
\]
is a c.p.~map satisfying $\overline{(\mathfrak B \otimes \mathfrak D)\Phi(\mathfrak C(\mathsf U)) (\mathfrak B \otimes \mathfrak D)} = \mathfrak B(\mathsf U) \otimes \mathfrak D$.

Conversely, suppose that $\mathfrak B\otimes \mathfrak D$ has Property (UBS) with respect to $\mathfrak C$. 
Clearly $\mathfrak B \otimes \mathfrak D$ is $\sigma$-unital since it has a countable approximate identity, so we may find $\tilde \Phi \colon \mathfrak C \to \mathfrak B \otimes \mathfrak D$ as in Remark \ref{r:ubsfactorB}. 
Let $\eta$ be a state on $\mathfrak D$ as given by Lemma \ref{l:coolstate}. Define $\Phi \colon \mathfrak C \to \mathfrak B$ to be the composition
 \[
  \mathfrak C \xrightarrow{\tilde \Phi} \mathfrak B \otimes \mathfrak D \xrightarrow{id_\mathfrak{B} \otimes \eta} \mathfrak B.
 \]

 Let $\mathfrak J_\mathsf{U} := \overline{\mathfrak B \Phi(\mathfrak C(\mathsf U)) \mathfrak B}$. Since 
\[
\Phi(\mathfrak C(\mathsf U)) = (id_\mathfrak{B} \otimes \eta)(\tilde \Phi(\mathfrak C(\mathsf U))) \subseteq (id_\mathfrak{B} \otimes \eta)(\mathfrak B(\mathsf U) \otimes \mathfrak D) = \mathfrak B(\mathsf U),
\]
it follows that $\mathfrak J_\mathsf{U} \subseteq \mathfrak B(\mathsf U)$. 
 By Lemma \ref{l:coolstate} any element in $\tilde \Phi(\mathfrak C(\mathsf U))$ will be in $\mathfrak J_{\mathsf U} \otimes \mathfrak D$. This implies that $\mathfrak B(\mathsf U) \otimes \mathfrak D \subseteq \mathfrak J_{\mathsf U} \otimes \mathfrak D$.
 It follows that $\mathfrak B(\mathsf U) = \mathfrak J_\mathsf{U}$ which finishes the proof. 
\end{proof}

The following proposition, which uses somewhat heavy machinery of Kirchberg and Rørdam, shows that almost all $\mathsf X$-$C^\ast$-algebras of interest have Property (UBS).

\begin{proposition}\label{p:manyproperty}
 Any separable, nuclear, upper semicontinuous $\mathsf X$-$C^\ast$-algebra has Property (UBS). 
 Moreover, we may choose that it has Property (UBS) with respect to a $\mathfrak C$, where the covering dimension of $\Prim \mathfrak C$ is at most 1.
\end{proposition}
Although we do not need the covering dimension of $\mathfrak C$ to be at most $1$ in this paper, the author believes that this could be important in future applications.
\begin{proof}
 Let $\mathfrak B$ be a separable, nuclear, upper semicontinuous $\mathsf X$-$C^\ast$-algebra.
 A $C^\ast$-subalgebra $\mathfrak C \subseteq \mathfrak B$ is called \emph{regular} if $(\mathfrak C \cap \mathfrak I) + (\mathfrak C \cap \mathfrak J) = \mathfrak C \cap (\mathfrak I + \mathfrak J)$,
 and if $\mathfrak C \cap \mathfrak I = \mathfrak C \cap \mathfrak J$ implies $\mathfrak I = \mathfrak J$ for all $\mathfrak I, \mathfrak J \in \mathbb I(\mathfrak B)$. 
 By \cite[Theorem 6.11]{KirchbergRordam-zero}\footnote{Note that the proof of \cite[Theorem 6.11]{KirchbergRordam-zero} does \emph{not} require any of the classification results of \cite{Kirchberg-non-simple} although other results in the paper do.
 Thus if one's goal is to reprove the results in \cite{Kirchberg-non-simple}, one may still use this result.},
 $\mathfrak B \otimes \mathcal O_2$ contains a regular, commutative $C^\ast$-subalgebra $\mathfrak C$ such that $\Prim \mathfrak C$ has covering dimension at most $1$. 
 Clearly $\mathfrak C$ is separable since $\mathfrak B$ is. Equip $\mathfrak C$ with the action of $\mathsf X$ given by
 $\mathfrak C(\mathsf U) = \mathfrak C \cap \mathfrak B(\mathsf U) \otimes \mathcal O_2$ for $\mathsf U\in \mathbb O(\mathsf X)$. 
  
 Since $\mathfrak C$ is a regular $C^\ast$-subalgebra of $\mathfrak B \otimes \mathcal O_2$, which is upper semicontinuous by Lemma \ref{l:actionstensorD}, $\mathfrak C$ is clearly upper semicontinuous. 
 
 Let $\mathfrak I \in \mathbb I(\mathfrak B \otimes \mathcal O_2)$ and let $\mathfrak J$ be the two-sided, closed ideal in $\mathfrak B \otimes \mathcal O_2$ generated by $\mathfrak C \cap \mathfrak I$. 
 Then $\mathfrak C \cap \mathfrak I = \mathfrak C \cap \mathfrak J$ which implies that $\mathfrak I = \mathfrak J$.
 Thus $\mathfrak B(\mathsf U) \otimes \mathcal O_2 = \overline{(\mathfrak B \otimes \mathcal O_2) \mathfrak C(\mathsf U) (\mathfrak B \otimes \mathcal O_2)}$ for all $\mathsf U\in \mathbb O(\mathsf X)$, so $\mathfrak B\otimes \mathcal O_2$ has Property (UBS) with respect to $\mathfrak C$. By Proposition \ref{p:ubstensorD}, $\mathfrak B$ has Property (UBS) with respect to $\mathfrak C$.
\end{proof}

\section{The ideal related lifting theorems}

In this section we prove $\mathsf X$-equivariant versions of the Choi--Effros lifting theorem and of the Effros--Haagerup lifting theorem. As a consequence, we show that extensions of nuclear $\mathsf X$-$C^\ast$-algebras have $\mathsf X$-equivariant c.p.~splittings, as long as the actions on the ideal and the quotient are sufficiently nice. Such results are closely related to ideal related $KK$-theory.

If $\mathfrak B$ is a $C^\ast$-algebra, and $\mathfrak J$ is a closed, two-sided ideal in $\mathfrak B$, then there are induced ideals in the multiplier algebra and the corona algebra, given by
\begin{eqnarray*}
\multialg{\mathfrak B, \mathfrak J} &=& \{ x \in \multialg{\mathfrak B} : x\mathfrak B \subseteq \mathfrak J\}, \\
\corona{\mathfrak B, \mathfrak J} &=& \pi(\multialg{\mathfrak B, \mathfrak J})
\end{eqnarray*}
where $\pi \colon \multialg{\mathfrak B} \to  \corona{\mathfrak B}$ is the quotient map. 

If $\mathfrak B$ is a stable $C^\ast$-algebra, then there exist isometries $s_1,s_2,\dots \in \multialg{\mathfrak B}$ such that $\sum_{k=1}^\infty s_k s_k^\ast$ converges strictly to $1_{\multialg{\mathfrak B}}$. By an \emph{infinite repeat} $x_\infty$ of an element $x \in \multialg{\mathfrak B}$, we mean $x_\infty = \sum_{k=1}^\infty s_k x s_k^\ast$, for some $s_1,s_2,\dots$ as above. Infinite repeats are unique up to unitary equivalence. In fact, if $t_1,t_2,\dots\in \multialg{\mathfrak B}$ are also isometries as above, then $u = \sum_{k=1}^\infty s_kt_k^\ast$ is a unitary in $\multialg{\mathfrak B}$ satisfying $u^\ast (\sum_{k=1}^\infty s_k x s_k^\ast) u = \sum_{k=1}^\infty t_k x t_k^\ast$.

\begin{lemma}\label{l:fullproj}
Let $\mathfrak J$ be a $\sigma$-unital ideal in a stable $C^\ast$-algebra $\mathfrak B$. Then $\multialg{\mathfrak B, \mathfrak J}$ contains a (norm-)full projection $P$.
\end{lemma}
\begin{proof}
As $\mathfrak J$ is an essential ideal in $\multialg{\mathfrak B, \mathfrak J}$ there is an induced embedding $\iota \colon \multialg{\mathfrak B, \mathfrak J} \hookrightarrow \multialg{\mathfrak J}$. The image of $\iota$ is easily seen to be a hereditary $C^\ast$-subalgebra of $\multialg{\mathfrak B}$. In fact, let $x_1,x_2\in \multialg{\mathfrak B,\mathfrak J}$ and $y\in \multialg{\mathfrak J}$. We define a multiplier $z \in \multialg{\mathfrak B, \mathfrak J}$ by
\[
z b := x_1(y(x_2b)), \quad bz := ((b x_1)y)x_2, \qquad b\in \mathfrak B.
\]
Then $\iota(z) = \iota(x_1) y \iota(x_2)$, so $\iota(\multialg{\mathfrak B, \mathfrak J})$ is hereditary in $\multialg{\mathfrak J}$.

As $\mathfrak B$ is stable, $\mathfrak B \cong \mathfrak B \otimes \ell^2(\mathbb N)$, as Hilbert $\mathfrak B$-modules. By Kasparov's stabilisation theorem \cite[Theorem 2]{Kasparov-Stinespring}, the Hilbert $\mathfrak B$-module $\mathfrak J \oplus \mathfrak B$ is isomorphic to $\mathfrak B$. Thus there is a projection $Q \in \mathbb B(\mathfrak B) = \multialg{\mathfrak B}$ (corresponding to $1\oplus 0 \in \mathbb B(\mathfrak J \oplus \mathfrak B)$), such that $Q \mathfrak B \cong \mathfrak J$ as Hilbert $\mathfrak B$-modules. As $\langle Q\mathfrak B , Q \mathfrak B \rangle = \mathfrak J$, it follows that $Q \mathfrak B Q = \mathbb K(Q\mathfrak B)$ is full in $\mathfrak J$.

Let $P$ be an infinite repeat of $Q$ in $\multialg{\mathfrak B}$. Clearly $P \in \multialg{\mathfrak B, \mathfrak J}$, and it is easy to see\footnote{The embedding $\iota$ extends to a strictly continuous, unital $\ast$-homomorphism $\iota \colon \multialg{\mathfrak B} \to \multialg{\mathfrak J}$, and if $s_1,s_2,\dots$ are isometries in $\multialg{\mathfrak B}$ defining an infinite repeat, then $\iota(s_1),\iota(s_2),\dots$ are isometries in $\multialg{\mathfrak J}$ which induce an infinite repeat by strict continuity.} that $\iota(P)$ is an infinite repeat of $\iota(Q)$ in $\multialg{\mathfrak J}$. Thus it follows from a result of Brown \cite[Lemma 2.5]{Brown-stableiso}, that $\iota(P)$ is Murray von-Neumann equivalent to $1_{\multialg{\mathfrak J}}$. As $\iota(\multialg{\mathfrak B, \mathfrak J})$ is a hereditary $C^\ast$-subalgebra of $\multialg{\mathfrak J}$, it follows that $P$ is full in $\multialg{\mathfrak B, \mathfrak J}$.
\end{proof}

For a positive element $x$ in a $C^\ast$-algebra, we let $(x-\epsilon)_+ := g_\epsilon(x)$ defined by functional calculus, where $g_\epsilon \colon [0,\infty) \to [0,\infty)$ is given by $g_\epsilon(t) = \max\{ 0, t-\epsilon\}$.

\begin{lemma}\label{l:almostfullideals}
Let $\mathfrak B$ be a separable, stable $C^\ast$-algebra, $x\in \multialg{\mathfrak B}$ be a positive element, and let $x_\infty$ denote an infinite repeat of $x$. For any $\epsilon >0$, let $\mathfrak J_\epsilon = \overline{\mathfrak B (x-\epsilon)_+ \mathfrak B}$. Then the ideal $\overline{\multialg{\mathfrak B}x_\infty \multialg{\mathfrak B}}$ contains the ideal $\multialg{\mathfrak B, \mathfrak J_\epsilon}$.
\end{lemma}
\begin{proof}
Fix an $\epsilon >0$, and let $f_\epsilon\colon [0,\infty) \to [0,\infty)$ be the continuous function
\[
f_\epsilon(t) = \left\{ \begin{array}{ll} 0, & t=0 \\ 1, & t\geq \epsilon \\ \textrm{affine, } & 0 \leq t \leq \epsilon. \end{array} \right.
\]
Let $y = (x-\epsilon)_+$ and $\mathfrak J_\epsilon=\overline {\mathfrak By\mathfrak B}$. As $\mathfrak B$ is separable, $\mathfrak J_\epsilon$ is $\sigma$-unital, so we may fix a full projection $P\in \multialg{\mathfrak B, \mathfrak J_\epsilon}$ by Lemma \ref{l:fullproj}. Fix a strictly positive element $h\in \mathfrak B$ of norm 1. We wish to recursively construct integers $n_1 < n_2 < \dots < n_k$ and $a_1,\dots, a_{n_k} \in \mathfrak B$ such that if $z_k = \sum_{i=1}^{n_k} a_i^\ast y a_i$ then $\| z_k \| \leq 1$ and $\|(1-z_k)Ph\| < 1/k$.

Suppose we have constructed such up to the stage $k$. As $(1-z_k)^{1/2}Ph \in \mathfrak J_\epsilon$, we may pick a positive element $e\in \mathfrak J_\epsilon$ with $\| e \| < 1$, such that $\| (1-e) (1-z_k)^{1/2} Ph \| < 1/(k+1)$. Let $\delta>0$ be small enough so that $\| e \| + \delta \leq 1$ and $\| (1-e) (1-z_k)^{1/2} Ph \| +\delta< 1/(k+1)$. We may find $m\in \mathbb N$, and $c_1,\dots, c_m\in \mathfrak B$ such that $z' := \sum_{i=1}^m c_i^\ast y c_i \approx_\delta e$. In particular, $\| z' \| \leq 1$ and 
\[
\| (1-z_k)^{1/2} (1-z') (1-z_k)^{1/2} Ph\| < 1/(k+1).
\]
Letting $n_{k+1} = n_k + m$ and $a_{n_k + i} = c_i(1-z_k)^{1/2}$ does the trick.

Then $\| hP( 1- \sum_{i=1}^{n_k} a_i^\ast y a_i) P h\|  \to 0$ for $k \to \infty$. As $hP(1-\sum_{i=1}^n a_i^\ast y a_i) Ph$ is monotonely decreasing and has a subsequence tending to $0$, the sequence itself will tend to zero for $n\to \infty$. Thus $\| (P- \sum_{i=1}^n b_i^\ast y b_i)h\| \to 0$ for $n\to \infty$ where $b_i := a_i P$. As $h$ is strictly positive it follows that $\sum_{i=1}^\infty b_i^\ast y b_i$ converges strictly to $P$.

Let $s_1,s_2,\dots$ be isometries in $\multialg{\mathfrak B}$ such that $\sum_{i=1}^\infty s_is_i^\ast$ converges strictly to $1_{\multialg{\mathfrak B}}$. Then (up to unitary equivalence) $x_\infty = \sum_{i=1}^\infty s_i x s_i^\ast$, and $f_\epsilon (x_\infty) = \sum_{i=1}^\infty s_i f_\epsilon(x) s_i^\ast$. 

Define the element $d = \sum_{i = 1}^\infty s_i y^{1/2} b_i$ (strict convergence). We check that this is well-defined, i.e.~that $\sum_{i = 1}^\infty s_i y^{1/2} b_i$ converges strictly. To see that $dh$ converges, note that
\[
\| \sum_{i=n}^m s_i y^{1/2} b_i h\| = \| h \sum_{i=n}^m b_i^\ast y b_i h \|^{1/2} \to 0, \quad \text{for } n,m \to \infty.
\]
Thus $\sum_{i = 1}^n s_i y^{1/2} b_i h$ is a Cuachy sequence, and thus converges. As $h$ is strictly positive, $\sum_{i = 1}^\infty s_i y^{1/2} b_i b$ converges for every $b\in \overline{h\mathfrak B} = \mathfrak B$.

Similarly, let $h_0 := \sum_{k=1}^\infty 2^{-k} s_k h s_k^\ast$. We get that
\[
\| h_0 \sum_{k=n}^m s_k y^{1/2} b_k \| = \| \sum_{k=n}^{m} 2^{-k} s_k h y^{1/2} b_i\| \leq \sum_{k=n}^m 2^{-k} \| h\| \|b_i^\ast y b_i\|^{1/2} \leq \sum_{k=n}^m 2^{-k} \|h\| \to 0
\]
for $n,m\to \infty$. Here we used that $b_i^\ast y b_i \leq \sum_{k=1}^{\infty} b_k^\ast y b_k \leq 1$. As above, $b \sum_{i = 1}^\infty s_i y^{1/2} b_i$ converges for $b\in \overline{B h_0}$. Thus, if $h_0$ is strictly positive, it will follow that $d$ is well-defined.

To see that $h_0$ is strictly positive, let $\delta>0$, pick $N\in \mathbb N$ such that $\sum_{j,k=1}^\infty s_j s_j^\ast h s_ks_k^\ast \approx_{\delta} \sum_{j,k=1}^N s_j s_j^\ast h s_ks_k^\ast$, and pick $c_{j,k}\in \mathfrak B$ such that $c_{j,k}h  \approx_{\delta/N^2} s_j^\ast h s_k$, which is doable as $h$ is strictly positive. Then 
\begin{eqnarray*}
h &=& \left( \sum_{j=1}^\infty s_js_j^\ast \right) h \left( \sum_{k=1}^\infty s_ks_k^\ast\right) \\
&=& \sum_{j,k=1}^\infty s_js_j^\ast h s_ks_k^\ast \\
&\approx_{\delta}& \sum_{j,k=1}^N s_j s_j^\ast h s_k s_k^\ast \\
&\approx_{\delta}& \sum_{j,k=1}^N s_j c_{j,k}h s_k^\ast \\
&=& \left(\sum_{j,k=1}^N 2^j s_jc_{j,k} s_k^\ast\right) h_0.
\end{eqnarray*}
As $\delta>0$ was arbitrary, it follows that $h\in \overline{\mathfrak Bh_0}$ which implies that $h_0$ is strictly positive, since $h$ is strictly positive. Hence $d$ is well-defined.

Then, as $f_\epsilon (x) y^{1/2} = y^{1/2}$, we get
\begin{eqnarray*}
d^\ast f_\epsilon (x_\infty) d &=& \left(\sum_{j=1}^\infty b_i^\ast y^{1/2} s_j^\ast \right) \left( \sum_{k=1}^\infty s_k f_\epsilon(x) s_k^\ast \right) \left( \sum_{l=1}^\infty s_l y^{1/2} b_l\right) \\
&=&  \sum_{i=1}^\infty b_i^\ast y^{1/2} f_\epsilon(x) y^{1/2} b_i \\
&=& \sum_{i=1}^\infty b_i^\ast y b_i = P.
\end{eqnarray*}
As $P$ was full in $\multialg{\mathfrak B,\mathfrak J_\epsilon}$, and as $x_\infty$ and $f(x_\infty)$ generate the same ideal in $\multialg{\mathfrak B}$, it follows that $\multialg{\mathfrak B, \mathfrak J_\epsilon} \subseteq \overline{\multialg{\mathfrak B}x_\infty \multialg{\mathfrak B}}$.
\end{proof}

Given any c.p.~map $\phi \colon \mathfrak A \to \mathfrak B$ with $\mathfrak A$ separable, and $\mathfrak B$ $\sigma$-unital and stable, Kasparov showed in \cite[Theorem 3]{Kasparov-Stinespring} that there is a Stinespring-type dilation, in the sense that there is a $\ast$-homomorphism $\Phi \colon \mathfrak A \to \multialg{\mathfrak B}$ and an element $V\in \multialg{\mathfrak B}$ such that $V^\ast \Phi(-) V = \phi$. The pair $(\Phi,V)$ is called the \emph{Kasparov--Stinespring dilation} of $\phi$, and the construction could be done as follows:

Construct the (right) Hilbert $\mathfrak B$-module $E:= \mathfrak A \otimes_\phi \mathfrak B$, by defining a pre-inner product on the algebraic tensor product $\mathfrak A \otimes_\mathbb{C} \mathfrak B$ given on elementary tensors by $\langle a_1 \otimes b_1, a_2 \otimes b_2\rangle =b_1^\ast\phi(a_1^\ast a_2) b_2$, quotienting out lenght zero vectors, and completing to obtain $E$. Let $\tilde \Phi \colon \mathfrak A \to \mathbb B(E) \subseteq \mathbb B(E\oplus \mathfrak B)$ be the $\ast$-homomorphism given by left multiplication on the left tensors. As $\mathfrak B$ is stable, $\mathfrak B \cong \mathfrak B \otimes \ell^2(\mathbb N)$ as Hilbert $\mathfrak B$-modules. Thus, by Kasparov's stabilisation theorem \cite[Theorem 2]{Kasparov-Stinespring}, there is a unitary $u\in \mathbb B(\mathfrak B , E\oplus \mathfrak B)$. We let
\[
\Phi := u^\ast \tilde \Phi(-) u \colon \mathfrak A \to \mathbb B(\mathfrak B) = \multialg{\mathfrak B}.
\]
If $W \in \mathbb B(\mathfrak B , E\oplus \mathfrak B)$ is the adjointable operator $W(b) = (1\otimes b, 0)$, and $V := u^\ast W$, then $V^\ast \Phi(-) V = \phi$.

Whenever $\mathsf X$ acts on $\mathfrak B$, there is an induced action on $\multialg{\mathfrak B}$ and $\corona{\mathfrak B}$ given by
\[
\multialg{\mathfrak B}(\mathsf U) := \multialg{\mathfrak B, \mathfrak B(\mathsf U)}, \qquad \corona{\mathfrak B}(\mathsf U) := \corona{\mathfrak B, \mathfrak B(\mathsf U)}.
\]

\begin{lemma}\label{l:dilationXnuc}
Let $\mathfrak A$ be a separable, exact $\mathsf X$-$C^\ast$-algebra, and $\mathfrak B$ be a $\sigma$-unital, stable $\mathsf X$-$C^\ast$-algebra. Suppose that $\phi \colon \mathfrak A \to \mathfrak B$ is a nuclear, $\mathsf X$-equivariant c.p.~map, and let $(\Phi,V)$ be the Kasparov--Stinespring dilation constructed above. Then $\Phi \colon \mathfrak A \to \multialg{\mathfrak B}$ is nuclear and $\mathsf X$-equivariant.
\end{lemma}
\begin{proof}
An element $x \in \multialg{\mathfrak B}$ is in $\multialg{\mathfrak B, \mathfrak J}$ if and only if $b^\ast x b \in \mathfrak J$ for every $b\in \mathfrak B$.
Thus $\Phi$ is $\mathsf X$-equivariant if and only if $b^\ast \Phi(-) b\colon \mathfrak A \to \mathfrak B$ is $\mathsf X$-equivariant for every $b\in \mathfrak B$. Morever, as $\mathfrak A$ is exact and $\mathfrak B$ is $\sigma$-unital, $\Phi$ is nuclear if and only if it $\Phi$ is weakly nuclear, i.e.~the maps $b^\ast \Phi(-) b$ are nuclear for every $b\in \mathfrak B$, by Proposition \ref{p:exactvsweaklynuc}. Thus it suffices to prove that $b^\ast \Phi (-) b$ is nuclear and $\mathsf X$-equivariant for every $b\in \mathfrak B$. Clearly it suffices to check this latter condition only on for $b$ in dense subset of $\mathfrak B$.

 Note that $b^\ast \Phi(-) b = b^\ast u^\ast (1\oplus 0) \tilde\Phi(-) (1\oplus 0) u b$ for any $b\in \mathfrak B$.
As any element $(1\oplus 0)u b\in \mathbb K(\mathfrak B, E \oplus \mathfrak B)$ can be approximated\footnote{We let $\theta_{x,y} \in \mathbb K(F,F')$ denote the ``rank 1'' operator $\theta_{x,y}(z) = x\langle y,z\rangle$, for $x\in F'$ and $y,z\in F$ ($F$ and $F'$ Hilbert modules), and recall that any element in $\mathbb K(F,F')$ can be approximated by sums of such ``rank 1'' operators.} by an element of the form $T= \sum_{i=1}^n \theta_{(x_i,0),c_i}$, where $x_i = \sum_{k=1}^{m_i} a_k^{(i)} \otimes b_k^{(i)} \in E$ (as such elements are dense in $E$) and $c_i \in \mathfrak B$. Thus it suffices to check that $T^\ast \tilde \Phi(-) T$ is nuclear and $\mathsf X$-equivariant for such $T\in \mathbb K(\mathfrak B , E \oplus \mathfrak B)$.

Observe, that
\begin{equation}\label{eq:innerproduct}
\langle (x_i,0), \tilde \Phi(-) (x_j,0)\rangle = \sum_{k=1}^{m_i} \sum_{l=1}^{m_j} \langle a_k^{(i)} \otimes b_k^{(i)}, (-)a_l^{(j)} \otimes b_l^{(j)} \rangle_E = \sum_{k=1}^{m_i} \sum_{l=1}^{m_j} b_k^{(i) \ast} \phi(a_k^{(i)\ast} (-) a_l^{(j)}) b_l^{(j)}.
\end{equation}
An easy computation shows that $\theta_{x,y}^\ast S \theta_{z,w}= \theta_{y\langle x,Sz\rangle, w}$, so
\begin{eqnarray*}
T^\ast \tilde \Phi(-) T &=& \sum_{i,j=1}^n \theta_{(x_i,0),c_i}^\ast \tilde\Phi(-) \theta_{(x_j,0),c_j} \\
&=& \sum_{i,j=1}^n \theta_{c_i \langle (x_i,0) , \tilde \Phi(-) (x_j,0)\rangle, c_j} \\
&\stackrel{\eqref{eq:innerproduct}}{=}& \sum_{i,j=1}^n \sum_{k=1}^{m_i} \sum_{l=1}^{m_j} \theta_{c_i  b_k^{(i) \ast} \phi(a_k^{(i)\ast} (-) a_l^{(j)}) b_l^{(j)} , c_j}.
\end{eqnarray*}
Under the canonical identification of $\mathbb K(\mathfrak B)$ and $\mathfrak B$, the ``rank 1'' operator $\theta_{d_1,d_2}$ corresponds to $d_1 d_2^\ast$. Thus, the map above, under this identification, is exactly the map
\[
 T^\ast \tilde \Phi(-) T = \sum_{i,j=1}^n \sum_{k=1}^{m_i} \sum_{l=1}^{m_j} c_i  b_k^{(i) \ast} \phi(a_k^{(i)\ast} (-) a_l^{(j)}) b_l^{(j)} c_j^\ast \colon \mathfrak A \to \mathfrak B.
\]
As $\phi$ is nuclear and $\mathsf X$-equivariant, and as the set of nuclear, $\mathsf X$-equivariant c.p.~maps is a closed operator convex cone, it follows that the above map is nuclear and $\mathsf X$-equivariant from Definition \ref{d:opconcon} (3). As $T$ was chosen arbitrarily in a dense subset of $(1\oplus 0)\mathbb K(\mathfrak B, E \oplus \mathfrak B)$, and as $\tilde \Phi = (1\oplus 0) \tilde \Phi(-) (1\oplus 0)$, it follows that $b^\ast \Phi(-) b$ is nuclear and $\mathsf X$-equivariant for all $b\in \mathfrak B$. As seen above, this implies that $\Phi$ is nuclear and $\mathsf X$-equivariant.
\end{proof}

\begin{proposition}\label{p:liftfromcorona}
Let $\mathsf X$ be a topological space, $\mathfrak A$ be a separable, exact, lower semicontinuous $\mathsf X$-$C^\ast$-algebra, and let $\mathfrak B$ be a separable $\mathsf X$-$C^\ast$-algebra with property (UBS). Then any nuclear, $\mathsf X$-equivariant c.p.~map $\eta \colon \mathfrak A \to \corona{\mathfrak B}$ lifts to a nuclear, $\mathsf X$-equivariant c.p.~map $\tilde \eta \colon \mathfrak A \to \multialg{\mathfrak B}$.
\end{proposition}
\begin{proof}
We start by proving the result under the additional assumption that $\mathfrak B$ is stable.
Let $\mathscr C$ denote the set of all nuclear, $\mathsf X$-equivariant c.p.~maps $\mathfrak A \to \multialg{\mathfrak B}$, which is a closed operator convex cone. By Proposition \ref{p:liftingconeversion}, it suffices to show that 
\begin{equation}\label{eq:liftcorona}
\eta(a) \in \pi\left( \overline{\multialg{\mathfrak B}\{ \psi(a) : \psi \in \mathscr C\} \multialg{\mathfrak B}} \right)
\end{equation}
for every positive $a\in \mathfrak A$. Fix $a\in \mathfrak A_+$. By Proposition \ref{p:Xfullmaps}, there are nuclear, $\mathsf X$-equivariant c.p.~maps $\phi_n \colon \mathfrak A \to \mathfrak B$ for $n\in \mathbb N$, such that $\mathfrak B(\mathsf U_{(a-1/n)_+}) = \overline{\mathfrak B \phi_n((a-1/n)_+)) \mathfrak B}$ for every $n\in \mathbb N$. We may assume that each map $\phi_n$ is contractive. Let $\phi := \sum_{n=1}^\infty 2^{-n} \phi_n$ which is nuclear and $\mathsf X$-equivariant, and for which $\mathfrak B(\mathsf U_{(a-1/n)_+}) = \overline{\mathfrak B \phi((a-1/n)_+) \mathfrak B}$ for $n\in \mathbb N$.

Let $(\Phi,V)$ be the Kasparov--Stinespring dilation of $\phi$. By Lemma \ref{l:dilationXnuc}, $\Phi \colon \mathfrak A \to \multialg{\mathfrak B}$ is a nuclear, $\mathsf X$-equivariant $\ast$-homomorphism. We get $\overline{\mathfrak B \Phi((a-1/n)_+) \mathfrak B} \subseteq \mathfrak B(\mathsf U_{(a-1/n)_+})$ since $\Phi$ is $\mathsf X$-equivariant. Also, 
\[
\overline{\mathfrak B \Phi((a-1/n)_+) \mathfrak B} \subseteq \overline{\mathfrak B V^\ast \Phi((a-1/n)_+) V\mathfrak B} = \overline{\mathfrak B \phi((a-1/n)_+) \mathfrak B} = \mathfrak B(\mathsf U_{(a-1/n)_+}),
\] 
so it follows that $\mathfrak B(\mathsf U_{(a-1/n)_+}) = \overline{\mathfrak B \Phi((a-1/n)_+) \mathfrak B}$ for each $n\in \mathbb N$.

Let $s_1,s_2,\dots\in \multialg{\mathfrak B}$ be isometries such that $\sum_{k=1}^\infty s_k s_k^\ast$ converges strictly to $1_{\multialg{\mathfrak B}}$, and let $\Phi_\infty = \sum_{k=1}^\infty s_k \Phi(-) s_k^\ast$ (convergence strictly). As $\mathfrak A$ is exact, and
\[
b^\ast \Phi_\infty(-) b = \sum_{k=1}^\infty b^\ast s_k \Phi(-) s_k^\ast b \colon \mathfrak A \to \mathfrak B
\]
is nuclear and $\mathsf X$-equivariant for all $b\in \mathfrak B$, it follows from Proposition \ref{p:exactvsweaklynuc} that $\Phi_\infty$ is nuclear and (clearly) $\mathsf X$-equivariant.

Recall, that our goal is to show \eqref{eq:liftcorona}. It suffices to show that 
\[
\eta((a-1/n)_+) \in \pi(\overline{\multialg{\mathfrak B} \Phi_\infty(a) \multialg{\mathfrak B}}),
\]
for every $n \in \mathbb N$. As $\eta$ is $\mathsf X$-equivariant, 
\[
\eta((a-1/n)_+) \in \corona{\mathfrak B}(\mathsf U_{(a-1/n)_+}) = \pi(\multialg{\mathfrak B, \mathfrak B(\mathsf U_{(a-1/n)_+})}).
\]
So it suffices to show that $\multialg{\mathfrak B, \mathfrak B(\mathsf U_{(a-1/n)_+})} \subseteq \overline{\multialg{\mathfrak B} \Phi_\infty(a) \multialg{\mathfrak B}}$ for every $n\in \mathbb N$. However, as
\[
\overline{\mathfrak B (\Phi(a) - 1/n)_+ \mathfrak B} = \overline{\mathfrak B \Phi((a-1/n)_+) \mathfrak B} = \mathfrak B(\mathsf U_{(a-1/n)_+}),
\]
this follows from Lemma \ref{l:almostfullideals}, and finishes the proof, under the assumption that $\mathfrak B$ is stable.

If $\mathfrak B$ is not stable, consider the nuclear, $\mathsf X$-equivariant map $\eta \otimes e_{11} \colon \mathfrak A \to \corona{\mathfrak B} \otimes \mathbb K \hookrightarrow \corona{\mathfrak B \otimes \mathbb K}$. By what we proved above, this lifts to a nuclear, $\mathsf X$-equivariant map $\tilde \eta' \colon \mathfrak A \to \multialg{\mathfrak B\otimes \mathbb K}$. Let $P=1_{\multialg{\mathfrak B}} \otimes e_{11}$. The map $\tilde \eta := P \tilde \eta' (-)P \colon \mathfrak A \to \multialg{\mathfrak B} \otimes e_{11} = \multialg{\mathfrak B}$ is a nuclear and $\mathsf X$-equivariant lift of $\eta$.
\end{proof}

An extension of $\mathsf X$-$C^\ast$-algebras is a short exact sequence $0 \to \mathfrak B \to \mathfrak E \to \mathfrak D \to 0$ in the category of $\mathsf X$-$C^\ast$-algebras such that for every open subset $\mathsf U$ of $\mathsf X$, the sequence
$0 \to \mathfrak B(\mathsf U) \to \mathfrak E(\mathsf U) \to \mathfrak D(\mathsf U) \to 0$ is a short exact sequence.

When $0\to \mathfrak B \to \mathfrak E \to \mathfrak D \to 0$ is an extension of $C^\ast$-algebras, and $\tau \colon \mathfrak D \to \corona{\mathfrak B}$ is the Busby map, we can construct the pull-back
\[
 \multialg{\mathfrak B} \oplus_{\corona{\mathfrak B}} \mathfrak D := \{ (x,d) \in \multialg{\mathfrak B} \oplus \mathfrak D : \pi(x) = \tau(d)\}.
\] 
It is well-known that $(\sigma,p)\colon \mathfrak E \xrightarrow{\cong} \multialg{\mathfrak B}\oplus_{\corona{\mathfrak B}} \mathfrak D$ is an isomorphism, where $p\colon \mathfrak E \to \mathfrak D$ is the quotient map, and $\sigma\colon \mathfrak E \to \multialg{\mathfrak B}$ is the canonical $\ast$-homomorphism.

It was shown in \cite[Proposition 5.20]{GabeRuiz-absrep}, that $0 \to \mathfrak B \to \mathfrak E \to \mathfrak D \to 0$ is an extension of $\mathsf X$-$C^\ast$-algebras if and only if $\tau$ is $\mathsf X$-equivariant and $(\sigma,p) \colon \mathfrak E \xrightarrow \cong \multialg{\mathfrak B}\oplus_{\corona{\mathfrak B}} \mathfrak D$ is an isomorphism of $\mathsf X$-$C^\ast$-algebras (i.e.~the map \emph{and its inverse} are $\mathsf X$-equivariant). Here we equipped the pull-back with the action
\[
(\multialg{\mathfrak B}\oplus_{\corona{\mathfrak B}} \mathfrak D)(\mathsf U) := (\multialg{\mathfrak B}\oplus_{\corona{\mathfrak B}} \mathfrak D) \cap (\multialg{\mathfrak B}(\mathsf U) \oplus \mathfrak D(\mathsf U)), \qquad \mathsf U\in \mathbb O(\mathsf X),
\]
which is well-defined whenever $\tau$ is $\mathsf X$-equivariant. We fill in the proof for completion.

\begin{proposition}\label{p:Xextensions}
Let $0 \to \mathfrak B \to \mathfrak E \to \mathfrak D \to 0$ be a short exact sequence in the category of $\mathsf X$-$C^\ast$-algebras. The sequence is an extension of $\mathsf X$-$C^\ast$-algebras if and only if the Busby map $\tau$ is $\mathsf X$-equivariant and the canonical isomorphism $\mathfrak E \to \multialg{\mathfrak B}\oplus_{\corona{\mathfrak B}} \mathfrak D$ is an isomorphism of $\mathsf X$-$C^\ast$-algebras.
\end{proposition}
\begin{proof}
If the sequence is an extension of $\mathsf X$-$C^\ast$-algebras, then $\mathfrak B(\mathsf U) = \mathfrak B \cdot \mathfrak E(\mathsf U)$, and $\mathfrak D(\mathsf U) = p(\mathfrak E(\mathsf U))$, where $p\colon \mathfrak E \to \mathfrak D$ is the quotient map. It follows that $\sigma \colon \mathfrak E \to \multialg{\mathfrak B}$ is $\mathsf X$-equivariant and thus $\tau$ is also $\mathsf X$-equivariant. We have 
\begin{equation}\label{eq:pullbackU}
(\multialg{\mathfrak B}\oplus_{\corona{\mathfrak B}} \mathfrak D)(\mathsf U) = \multialg{\mathfrak B}(\mathsf U)\oplus_{\corona{\mathfrak B}(\mathsf U)} \mathfrak D(\mathsf U)  \stackrel{(\ast)}{=} \multialg{\mathfrak B(\mathsf U)} \oplus_{\corona{\mathfrak B(\mathsf U)}} \mathfrak D(\mathsf U),
\end{equation}
for $\mathsf U\in \mathbb O(\mathsf X)$, where $(\ast)$ is easily verified, e.g.~by uniqueness of pull-backs, and is left for the reader. It follows that the isomorphism $\mathfrak E \xrightarrow \cong \multialg{\mathfrak B}\oplus_{\corona{\mathfrak B}} \mathfrak D$ restricts to an isomorphism $\mathfrak E(\mathsf U) \xrightarrow \cong (\multialg{\mathfrak B}\oplus_{\corona{\mathfrak B}} \mathfrak D)(\mathsf U)$ for every $\mathsf U$. Thus it is an isomorphism of $\mathsf X$-$C^\ast$-algebras.

Conversely, suppose $\tau$ is $\mathsf X$-equivariant, and $\mathfrak E \to \multialg{\mathfrak B}\oplus_{\corona{\mathfrak B}} \mathfrak D$ is an isomorphism of $\mathsf X$-$C^\ast$-algebras. As \eqref{eq:pullbackU} holds, it follows that $0\to \mathfrak B(\mathsf U) \to \mathfrak E(\mathsf U) \to \mathfrak D(\mathsf U) \to 0$ is exact, so $0 \to \mathfrak B \to \mathfrak E \to \mathfrak D \to 0$ is an extension of $\mathsf X$-$C^\ast$-algebras.
\end{proof}

We can now prove an ideal related lifting theorem. Part $(i)$ in the theorem is an $\mathsf X$-equivariant version of the Choi--Effros lifting theorem \cite{ChoiEffros-lifting}, and part $(ii)$ is an $\mathsf X$-equivariant version of the Effros--Haagerup lifting theorem \cite{EffrosHaagerup-lifting}.

\begin{theorem}\label{t:Xlifting}
Let $\mathsf X$ be a topological space, let $0 \to \mathfrak B \to \mathfrak E \to \mathfrak D\to 0$ be an extension of $\mathsf X$-$C^\ast$-algebras, for which $\mathfrak B$ is separable and has Property (UBS) (in particular, $\mathfrak B$ could be upper semicontinuous and nuclear), and let $\mathfrak A$ be a separable, exact, lower semicontinuous $\mathsf X$-$C^\ast$-algebra. Let $\phi\colon \mathfrak A \to \mathfrak D$ be an $\mathsf X$-equivariant c.p.~map.
If one of the following hold:
\begin{itemize}
\item[$(i)$] $\phi$ is nuclear,
\item[$(ii)$] $\mathfrak E$ is exact and $\mathfrak B$ is nuclear,
\end{itemize}
then there exists an $\mathsf X$-equivariant c.p.~lift $\tilde \phi \colon \mathfrak A \to \mathfrak E$.
\end{theorem}
\begin{proof}
Let $\tau$ denote the Busby map of our given extension. By Proposition \ref{p:Xextensions}, $\tau$ is $\mathsf X$-equivariant and $\mathfrak E \xrightarrow \cong \multialg{\mathfrak B} \oplus_{\corona{\mathfrak B}} \mathfrak A$ is an isomorphism of $\mathsf X$-$C^\ast$-algebras. 

Suppose that $\tau \circ \phi$ is nuclear. As $\tau \circ \phi$ is $\mathsf X$-equivariant, we may lift $\tau \circ \phi$ to an $\mathsf X$-equivariant c.p.~map $\psi \colon \mathfrak A \to \multialg{\mathfrak B}$, by Proposition \ref{p:liftfromcorona}. The c.p.~map 
\[
\tilde \phi = (\psi,\phi) \colon \mathfrak A \to \multialg{\mathfrak B} \oplus_{\corona{\mathfrak B}} \mathfrak A \cong \mathfrak E
\]
is an $\mathsf X$-equivariant lift of $\phi$. So it suffices to show that $\tau \circ \phi$ is nuclear if either $(i)$ or $(ii)$ holds. If $\phi$ is nuclear (i.e.~$(i)$ holds), then $\tau \circ \phi$ is nuclear, as compositions of a nuclear c.p.~map with any c.p.~map is nuclear. If $\mathfrak E$ is exact and $\mathfrak B$ is nuclear (i.e.~$(ii)$ holds), then $\mathfrak D$ is exact, as quotients of exact $C^\ast$-algebras are exact \cite{Kirchberg-CAR}. Thus by Corollary \ref{c:exactext}, $\tau$ is nuclear and thus $\tau \circ \phi$ is nuclear. ``In particular'' follows from Proposition \ref{p:manyproperty}.
\end{proof}

A consequence of Theorem \ref{t:Xlifting} is the following result, which say that in most cases of interest, an extension of $\mathsf X$-$C^\ast$-algebra will be \emph{semisplit}, i.e.~it will have an $\mathsf X$-equivariant c.p.~splitting, as long as the quotient is lower semicontinuous.

\begin{theorem}\label{t:Xsplit}
Let $\mathsf X$ be a topological space, and $0 \to \mathfrak B \to \mathfrak E \to \mathfrak A\to 0$ be an extension of separable, nuclear $\mathsf X$-$C^\ast$-algebras. Suppose that $\mathfrak B$ is upper semicontinuous and $\mathfrak A$ is lower semicontinuous. Then there is an $\mathsf X$-equivariant c.p.~splitting $\mathfrak A \to \mathfrak E$.
\end{theorem}
\begin{proof}
Apply Theorem \ref{t:Xlifting}, with $\mathfrak A = \mathfrak D$ and $\phi = id_{\mathfrak A}$, to find the $\mathsf X$-equivariant c.p.~splitting.
\end{proof}

\begin{remark}
It is well-known that the above theorem fails if we remove the lower semicontinuity assumption of $\mathfrak A$. E.g., the extension $0 \to C_0((0,1]) \to C([0,1]) \to \mathbb C \to 0$ of $[0,1]$-$C^\ast$-algebras (with the obvious actions) can never have an $[0,1]$-equivariant c.p.~splitting (or even $[0,1]$-equivariant non-c.p.~splitting), as the only $[0,1]$-equivariant map $\mathbb C \to C([0,1])$ is the zero map.
\end{remark}


 \section{Comparing ideal related $KK$-theory and $E$-theory}

Recall, that a \emph{$C^\ast$-algebra over $\mathsf X$}, is an $\mathsf X$-$C^\ast$-algebra for which the action is finitely lower semicontinuous and upper semicontinuous.
In \cite{DadarlatMeyer-E-theory}, Dadarlat and Meyer construct $E$-theory for separable $C^\ast$-algebras over $\mathsf X$ when $\mathsf X$ is second countable. We will sketch the construction.

An \emph{asymptotic morphism} from $\mathfrak A$ to $\mathfrak B$ is a map $\phi\colon \mathfrak A \to C_b([0,\infty), \mathfrak B)$,
such that the composition of this map with the quotient map onto $\mathfrak B_\infty := C_b([0,\infty), \mathfrak B)/C_0([0,\infty), \mathfrak B)$, 
call this composition $\dot \phi$, is a $\ast$-homomorphism. If $\phi$ and $\phi'$ are asymptotic morphisms, we say that they are \emph{equivalent} if $\dot \phi = \dot \phi'$.
If $\mathfrak A$ and $\mathfrak B$ are $\mathsf X$-$C^\ast$-algebras, then $\mathfrak B_\infty$ is an $\mathsf X$-$C^\ast$-algebra by
\[
\mathfrak B_\infty(\mathsf U) = \frac{C_b([0,\infty), \mathfrak B(\mathsf U)) + C_0([0,\infty), \mathfrak B)}{C_b([0,\infty), \mathfrak B)}.
\]
We say that an asymptotic morphism $\phi$ is \emph{approximately $\mathsf X$-equivariant} if the induced $\ast$-homomorphism $\dot \phi$ is $\mathsf X$-equivariant. Note that this does not imply that the asymptotic morphism is $\mathsf X$-equivariant.
We say that two approximately $\mathsf X$-equivariant asymptotic morphism $\phi_0,\phi_1$ from $\mathfrak A$ to $\mathfrak B$ are \emph{homotopic} if there is an approximately $\mathsf X$-equivariant asymptotic morphism $\Phi$ from $\mathfrak A$ to
$C([0,1],\mathfrak B)$ such that $\mathrm{ev}_i \circ \Phi = \phi_i$ for $i=0,1$. We let $[[ \mathfrak A, \mathfrak B]]_\mathsf{X}$ denote the set of homotopy classes of approximately $\mathsf X$-equivariant asymptotic morphisms. 
 
For separable $C^\ast$-algebras $\mathfrak A$ and $\mathfrak B$ over $\mathsf X$, where $\mathsf X$ is second countable, we define
\[
E(\mathsf X; \mathfrak A, \mathfrak B) := [[ C_0(\mathbb R) \otimes \mathfrak A \otimes \mathbb K , C_0(\mathbb R) \otimes \mathfrak B \otimes \mathbb K]]_\mathsf{X}.
\]
This comes equipped with an abelian group structure, as well as a bilinear composition product. Thus $E(\mathsf X; -,-)$ is a bivariant functor from the category of separable $C^\ast$-algebras over $\mathsf X$ to the category of abelian groups.

Similarly, consider asymptotic morphisms $\phi\colon \mathfrak A \to C_b([0,\infty), \mathfrak B)$, such that $\phi$ is an $\mathsf X$-equivariant contractive c.p.~map. 
Note that these are actually $\mathsf X$-equivariant and not just approximately $\mathsf X$-equivariant. 
By again taking homotopies only of this form we may construct the set $[[\mathfrak A, \mathfrak B]]_\mathsf{X}^{cp}$ of homotopy classes of such asymptotic morphisms.

By \cite[Theorem 5.2]{DadarlatMeyer-E-theory}, when $\mathsf X$ is second countable, and $\mathfrak A$ and $\mathfrak B$ are separable $C^\ast$-algebras over $\mathsf X$, there is a natural isomorphism
\[
KK(\mathsf X; \mathfrak A, \mathfrak B) \cong [[ C_0(\mathbb R) \otimes \mathfrak A \otimes \mathbb K , C_0(\mathbb R) \otimes \mathfrak B \otimes \mathbb K]]_\mathsf{X}^{cp}.
\]

\begin{remark}
Although $C^\ast$-algebras over $\mathsf X$ are commonly thought of as the ``correct'' generalisation of $C^\ast$-algebras when one wants to incorporate ideal structure, 
there are given examples in \cite{GabeRuiz-absrep} of why it is not always convenient only to consider $C^\ast$-algebras over $\mathsf X$ instead of more general $\mathsf X$-$C^\ast$-algebras.
$E(\mathsf X)$-theory can easily be generalised to monotone upper semicontinuous $\mathsf X$-$C^\ast$-algebras (which will be important in future work by the author), however, 
for simplicity we will mainly work with $C^\ast$-algebras over $\mathsf X$ in this section.
\end{remark}

Recall that an $\mathsf X$-$C^\ast$-algebra (or a $C^\ast$-algebra over $\mathsf X$) is called \emph{continuous} if it is lower and upper semicontinuous.

\begin{theorem}\label{t:evskk}
Let $\mathsf X$ be second countable, and let $\mathfrak A$ and $\mathfrak B$ be separable, nuclear $C^\ast$-algebras over $\mathsf X$.
If $\mathfrak A$ continuous, then $E(\mathsf X; \mathfrak A, \mathfrak B) \cong KK(\mathsf X; \mathfrak A, \mathfrak B)$ naturally.
\end{theorem}
\begin{proof}
Let $\phi$ be an approximately $\mathsf X$-equivariant asymptotic morphism from $\mathfrak A$ to $\mathfrak B$, and
\[
 \dot \phi \colon \mathfrak A \to \mathfrak B_\infty := \frac{C_b([0,\infty),\mathfrak B)}{C_0([0,\infty), \mathfrak B)}
\]
be the induced $\mathsf X$-equivariant $\ast$-homomorphism. Consider the pull-back diagram
\[
\xymatrix{
0 \ar[r] & C_0([0,\infty),\mathfrak B) \ar@{=}[d] \ar[r] & \mathfrak E \ar[d]^\sigma \ar[r] & \mathfrak A \ar[d]^{\dot \phi} \ar[r] & 0 \\
0 \ar[r] & C_0([0,\infty),\mathfrak B) \ar[r] & C_b([0,\infty),\mathfrak B) \ar[r] & \mathfrak B_\infty \ar[r] & 0,
}
\]
and observe that the top row is an extension of $\mathsf X$-$C^\ast$-algebras.
By Lemma \ref{l:actionstensorD}, $C_0([0,\infty),\mathfrak B)$ is a $C^\ast$-algebra over $\mathsf X$.
Thus by Theorem \ref{t:Xsplit}, there is an $\mathsf X$-equivariant contractive c.p.~split $\psi \colon \mathfrak A \to \mathfrak E$.
It follows that $\sigma \circ \psi$ is an $\mathsf X$-equivariant, contractive c.p.~asymptotic morphism which is equivalent to $\phi$.
By replacing $\mathfrak B$ with $C([0,1],\mathfrak B)$, it follows that any approximately $\mathsf X$-equivariant asymptotic homotopy may be replaced by an
$\mathsf X$-equivariant, contractive c.p.~asymptotic homotopy. Thus
\[
 [[ \mathfrak A, \mathfrak B]]_\mathsf{X} = [[\mathfrak A, \mathfrak B]]_{\mathsf X}^{cp}.
\]
 Let $\mathfrak A_0 = C_0(\mathbb R, \mathfrak A)\otimes \mathbb K$ and $\mathfrak B_0 = C_0(\mathbb R,\mathfrak B) \otimes \mathbb K$, which are nuclear $C^\ast$-algebras over $\mathsf X$.
 By what we proved above and by \cite[Theorem 5.2]{DadarlatMeyer-E-theory} it follows that
 \[
  KK(\mathsf X; \mathfrak A, \mathfrak B) \cong [[ \mathfrak A_0 , \mathfrak B_0]]_{\mathsf X}^{cp} = [[ \mathfrak A_0 , \mathfrak B_0]]_{\mathsf X} = E(\mathsf X; \mathfrak A, \mathfrak B),
 \]
where the isomorphism is natural.
\end{proof}

\begin{remark}
 The proof above can easily be modified so that we only require that $\mathfrak B$ has Property (UBS) instead of being nuclear.
 Thus if $\mathsf X$ is finite or if it is locally compact and Hausdorff, we do not need nuclearity of $\mathfrak B$ in Theorem \ref{t:evskk}.
\end{remark}


\section{Absorption of strongly self-absorbing $C^\ast$-algebras}

In this section we give a few easy applications of Theorem \ref{t:evskk}. Using this result we can weaken the deep classification result of Kirchberg \cite{Kirchberg-non-simple}. A proof of this theorem can be found in \cite{Gabe-Oinfty} by the author.

A $C^\ast$-algebra is called \emph{strongly purely infinite}, if it has a certain comparability property defined by Kirchberg and Rørdam in \cite[Definition 5.1]{KirchbergRordam-absorbingOinfty}. As we do not need the exact definition in this paper, we simply mention that a separable, nuclear $C^\ast$-algebra $\mathfrak A$ is stronly purely infinite if and only if $\mathfrak A \otimes \mathcal O_\infty \cong \mathfrak A$, by \cite[Theorem 8.6]{KirchbergRordam-absorbingOinfty} and \cite[Corollary 3.2]{TomsWinter-ssa}.

Recall that an action $\mathbb O(\mathsf X) \to \mathbb I(\mathfrak A)$ is \emph{tight} if it is a lattice isomorphism.

\begin{theorem}[Cf.~Kirchberg]\label{t:Kirchberg}
Let $\mathfrak A$ and $\mathfrak B$ be separable, nuclear, stable, strongly purely infinite, tight $\mathsf X$-$C^\ast$-algebras. 
Then any invertible element in $E(\mathsf X; \mathfrak A, \mathfrak B)$ lifts to an $\mathsf X$-equivariant $\ast$-isomorphism $\mathfrak A \to \mathfrak B$.
\end{theorem}
\begin{proof}
By Theorem \ref{t:evskk}, $E(\mathsf X; \mathfrak A, \mathfrak B) \cong KK(\mathsf X; \mathfrak A, \mathfrak B)$ and $E(\mathsf X; \mathfrak B , \mathfrak A) \cong KK(\mathsf X; \mathfrak B , \mathfrak A)$ naturally. Thus any invertible element in $E(\mathsf X; \mathfrak A, \mathfrak B)$ lifts to an invertible element in $KK(\mathsf X; \mathfrak A, \mathfrak B)$, which in turn lifts to an $\mathsf X$-equivariant $\ast$-isomorphism $\mathfrak A \to \mathfrak B$ by a very deep theorem of Kirchberg \cite{Kirchberg-non-simple} (alternatively, see \cite[Theorem G]{Gabe-Oinfty}).
\end{proof}

It turns out (cf.~\cite[Theorem 4.6]{DadarlatMeyer-E-theory}) that ideal related $E$-theory is (a priori) much more well-behaved with respect to $K$-theory than ideal related $KK$-theory, and thus it should be easier to apply the above theorem for $K$-theoretic classification than the original result of Kirchberg.

As is costumary, we say that a separable $C^\ast$-algebra \emph{satisfies the UCT}, if it satisfies the universal coefficient theorem of Rosenberg and Schochet \cite{RosenbergSchochet-UCT}. This is equivalent to the $C^\ast$-algebra being $KK$-equivalent to a commutative $C^\ast$-algebra.

For any $\alpha \in E(\mathsf X; \mathfrak A, \mathfrak B)$ there is an induced element $\alpha_\mathsf{U} \in E(\mathfrak A(\mathsf U), \mathfrak B(\mathsf U))$. In particular, this also induces a homomorphism in $K$-theory $K_\ast(\alpha_\mathsf{U}) \colon K_\ast(\mathfrak A(\mathsf U)) \to K_\ast(\mathfrak B(\mathsf U))$. The following result of Dadarlat and Meyer gives a very effective way of determining when an $E(\mathsf X)$-element is invertible as a ``point-wise'' condition.

\begin{theorem}[\cite{DadarlatMeyer-E-theory}, Theorems 3.10 and 4.6]\label{t:ptwiseinv}
Let $\mathsf X$ be a second countable space, and let $\mathfrak A$ and $\mathfrak B$ be separable $C^\ast$-algebras over $\mathsf X$. An element $\alpha \in E(\mathsf X; \mathfrak A, \mathfrak B)$ is invertible if and only if the induced element $\alpha_\mathsf{U} \in E(\mathfrak A(\mathsf U), \mathfrak B(\mathsf U))$ is invertible for each $\mathsf U\in \mathbb O(\mathsf X)$.

In particular, if $\mathfrak A(\mathsf U)$ and $\mathfrak B(\mathsf U)$ satisfy the UCT of Rosenberg and Schochet for each $\mathsf U\in \mathbb O(\mathsf X)$, then $\alpha \in E(\mathsf X; \mathfrak A, \mathfrak B)$ is invertible if and only if $K_\ast(\alpha_\mathsf{U}) \colon K_\ast(\mathfrak A(\mathsf U)) \to K_\ast(\mathfrak B(\mathsf U))$ is an isomorphism for each $\mathsf U\in \mathbb O(\mathsf X)$.
\end{theorem}

\begin{definition}[Toms--Winter \cite{TomsWinter-ssa}]
A separable, unital $C^\ast$-algebra $\mathscr D$ is called \emph{strongly self-absorbing} if $\mathscr D \not \cong \mathbb C$ and if there exists an isomorphism $\phi \colon \mathscr D \to \mathscr D \otimes \mathscr D$ which is approximately unitarily equivalent to the $\ast$-homomorphism $id_\mathscr{D} \otimes 1_{\mathscr D} \colon \mathscr D \to \mathscr D \otimes \mathscr D$.
\end{definition}

The following are all known examples of strongly self-absorbing $C^\ast$-algebras: the Cuntz algebras $\mathcal O_2$ and $\mathcal O_\infty$, all UHF algebras of infinite type, the Jiang--Su algebra $\mathcal Z$, and any UHF algebra of infinite type tensor $\mathcal O_\infty$. Any strongly self-absorbing $C^\ast$-algebra that satisfies the UCT of Rosenberg and Schochet is one of the above. For more information, see \cite{Winter-qdquct} for a good overview.

\begin{proposition}\label{p:absorbssa}
Let $\mathfrak A$ be a separable, nuclear, strongly purely infinite $C^\ast$-algebra, and let $\mathscr D$ be a strongly self-absorbing $C^\ast$-algebra. Then $\mathfrak A \cong \mathfrak A \otimes \mathscr D$ if and only if $\mathfrak I$ and $\mathfrak I \otimes \mathscr D$ are $KK$-equivalent for every two-sided, closed ideal $\mathfrak I$ in $\mathfrak A$.
\end{proposition}
\begin{proof}
If $\mathfrak A \cong \mathfrak A \otimes \mathscr D$, then $\mathfrak I \cong \mathfrak I \otimes \mathscr D$ for every two-sided, closed ideal $\mathfrak I$ in $\mathfrak A$. In particular, $\mathfrak I$ and $\mathfrak I \otimes \mathscr D$ are $KK$-equivalent.

Suppose that $\mathfrak I$ and $\mathfrak I \otimes \mathscr D$ are $KK$-equivalent, and let $\alpha \in KK(\mathfrak I, \mathfrak I \otimes \mathscr D)$ be invertible. The Kasparov product (composition)
\[
\mathfrak I \xrightarrow{\alpha} \mathfrak I \otimes \mathscr D \xrightarrow{id_{\mathfrak I} \otimes id_{\mathscr D} \otimes 1_{\mathscr D}} \mathfrak I \otimes \mathscr D \otimes \mathscr D \xrightarrow{\alpha^{-1} \otimes id_{\mathscr D}} \mathfrak I \otimes \mathscr D
\]
is exactly $id_{\mathfrak I} \otimes 1_{\mathscr D} \colon \mathfrak I \to \mathfrak I\otimes \mathscr D$. Clearly $\alpha$ and $\alpha^{-1} \otimes id_{\mathscr D}$ are invertible.  By \cite[Theorem 2.2]{DadarlatWinter-KKssa}, $id_{\mathscr D} \otimes 1_{\mathscr D} \colon \mathscr D \to \mathscr D \otimes \mathscr D$ is \emph{asymptotically} unitarily equivalent to an isomorphism $\phi$ (as any strongly self-absorbing $C^\ast$-algebra is $K_1$-injective by \cite[Remark 3.3]{Winter-ssaZstable}). Thus $id_{\mathfrak A} \otimes id_{\mathscr D} \otimes 1_{\mathscr D}$ is invertible in $KK$-theory, and hence $id_{\mathfrak I} \otimes 1_{\mathscr D} \colon \mathfrak I \to \mathfrak I \otimes \mathscr D$ induces a $KK$-equivalence, as it is a composition of $KK$-equivalences.

Let $\mathsf X = \Prim \mathfrak A$. Equip $\mathfrak A\otimes \mathscr D$ with the action $\mathbb O(\mathsf X) \to \mathbb I(\mathfrak A \otimes \mathscr D)$ given by $(\mathfrak A \otimes \mathscr D)(\mathsf U) = \mathfrak A (\mathsf U) \otimes \mathscr D$. By \cite[Theorem 1.6]{TomsWinter-ssa}, $\mathscr D$ is simple and nuclear, and thus $\mathfrak A \otimes \mathscr D$ is a separable, nuclear, strongly purely infinite, tight $\mathsf X$-$C^\ast$-algebra. Note that $id_\mathfrak{A} \otimes 1_\mathscr{D} \colon \mathfrak A \to \mathfrak A \otimes \mathscr D$ is $\mathsf X$-equivariant, and thus induces an $E(\mathsf X)$-element $\alpha \in E(\mathsf X; \mathfrak A, \mathfrak A \otimes \mathscr D)$. As
\[
id_{\mathfrak A(\mathsf U)} \otimes 1_{\mathscr D} \colon \mathfrak A(\mathsf U) \to (\mathfrak A \otimes \mathbb K)(\mathsf U)\otimes \mathscr D
\]
induces an invertible $KK$-element, and thus also an invertible $E$-element, which is $\alpha_\mathsf{U}$, for every $\mathsf U\in \mathbb O(\mathsf X)$, it follows from Theorem \ref{t:ptwiseinv} that $\alpha$ is invertible. Using Theorem \ref{t:Kirchberg} and the fact that ideal-related $E$-theory is stable, we obtain an isomorphism $\mathfrak A \otimes \mathbb K \cong \mathfrak A \otimes \mathscr D \otimes \mathbb K$. By \cite[Corollary 3.2]{TomsWinter-ssa}, $\mathfrak A \cong \mathfrak A \otimes \mathscr D$.
\end{proof}

\begin{definition}
We say that an abelian group $G$ is \emph{uniquely $n$-divisible} for an integer $n\geq 2$, if for every $g\in G$ there is a unique element $h\in G$ such that $n\cdot h =g$.

We say that $G$ is \emph{uniquely divisible} if it is uniquely $n$-divisible for every $n\geq 2$.
\end{definition}

Note that an abelian group $G$ is uniquely $n$-divisible if and only if $G \cong G \otimes \mathbb Z[\tfrac{1}{n}]$.

For any $n \geq 2$ we let $M_{n^{\infty}} = M_n \otimes M_n \otimes \dots$ denote the UHF algebra of type $n^\infty$. We let $\mathcal Q = \bigotimes_{k\in \mathbb N} M_k$ denote the universal UHF algebra.

\begin{theorem}
Let $\mathfrak A$ be a separable, nuclear, strongly purely infinite $C^\ast$-algebra, for which every two-sided, closed ideal satisfies the UCT. For $n\geq 2$, it holds that $\mathfrak A \cong \mathfrak A \otimes M_{n^\infty}$ if and only if $K_\ast(\mathfrak I)$ is uniquely $n$-divisible for every two-sided, closed ideal $\mathfrak I$ in $\mathfrak A$.

In particular, $\mathfrak A \cong \mathfrak A \otimes \mathcal Q$ if and only if $K_\ast(\mathfrak I)$ is uniquely divisible for every two-sided, closed ideal $\mathfrak I$ in $\mathfrak A$.
\end{theorem}
\begin{proof}
If $\mathfrak A \cong \mathfrak A \otimes M_{n^\infty}$ then $\mathfrak I \cong \mathfrak I \otimes M_{n^\infty}$ for every two-sided, closed ideal $\mathfrak I$ in $\mathfrak A$, and thus
\[
K_i (\mathfrak I) \cong K_i (\mathfrak I \otimes M_{n^\infty}) \cong K_i (\mathfrak I) \otimes \mathbb Z[\tfrac{1}{n}]
\]
by the Künneth theorem \cite{Schochet-topII} for $i=0,1$. Hence $K_\ast(\mathfrak I)$ is uniquely $n$-divisible.

Conversely, suppose that $K_\ast(\mathfrak I)$ is uniquely $n$-divisible for every two-sided, closed ideal $\mathfrak I$ in $\mathfrak A$. Then, as above, $K_i(\mathfrak I) \cong K_i(\mathfrak I)\otimes \mathbb Z[\tfrac{1}{n}] \cong K_i(\mathfrak I \otimes M_{n^\infty})$ for $i=0,1$. As $\mathfrak I$ and $\mathfrak I \otimes M_{n^\infty}$ satisfy the UCT, it follows that $\mathfrak I$ and $\mathfrak I \otimes M_{n^\infty}$ are $KK$-equivalent. Thus $\mathfrak A \cong \mathfrak A \otimes M_{n^\infty}$ by Proposition \ref{p:absorbssa}. 

The ``in particular'' part follows since $\mathfrak A \cong \mathfrak A \otimes \mathcal Q$ if and only if $\mathfrak A \cong \mathfrak A \otimes M_{n^\infty}$ for every $n\geq 2$.
\end{proof}

Recall, that a separable $C^\ast$-algebra is \emph{$KK$-contractible} if it is $KK$-equivalent to $0$. Note that $\mathfrak A$ is $KK$-contractible if and only if it satisfies the UCT and $K_\ast(\mathfrak A) = 0$.

\begin{theorem}
Let $\mathfrak A$ be a separable, nuclear, strongly purely infinite $C^\ast$-algebra. Then $\mathfrak A \cong \mathfrak A \otimes \mathcal O_2$ if and only if every two-sided, closed ideal in $\mathfrak A$ is $KK$-contractible.
\end{theorem}
\begin{proof}
If $\mathfrak A \cong \mathfrak A \otimes \mathcal O_2$ then $\mathfrak I \cong \mathfrak I \otimes \mathcal O_2$ for every closed, two-sided ideal $\mathfrak I$ in $\mathfrak A$. It follows that $\mathfrak I$ is $KK$-contractible.

Conversely, if $\mathfrak I$ is $KK$-contractible for each two-sided, closed ideal $\mathfrak I$ in $\mathfrak A$, then $id_{\mathfrak I} \otimes 1_{\mathcal O_2} \colon \mathfrak I \to \mathfrak I \otimes \mathcal O_2$ induces a $KK$-equivalence. Hence $\mathfrak A \cong \mathfrak A \otimes \mathcal O_2$ by Proposition \ref{p:absorbssa}.
\end{proof}

\subsection*{Acknowledgement}
The author would like to thank Rasmus Bentmann, Søren Eilers, Ryszard Nest, and Efren Ruiz for helpful discussions, comments, and suggestions.

\end{document}